\documentclass[reqno,11pt]{amsart}

\usepackage{amssymb}
\usepackage{amsmath}
\usepackage{latexsym}
\usepackage{amsfonts}
\newtheorem{Pa}{Paper}[section]
\newtheorem{theorem}[Pa]{{\bf Theorem}}
\newtheorem{lemma}[Pa]{{\bf Lemma}}

\newtheorem{corollary}[Pa]{{\bf Corollary}}

\newtheorem{Rk}[Pa]{{\bf Remark}}
\newtheorem{proposition}[Pa]{{\bf Proposition}}

\newtheorem{Ex}[Pa]{{\bf Exercise}}

\numberwithin{equation}{section}

\newcommand{\sbm}[1]{\left[\begin{smallmatrix} #1
                \end{smallmatrix}\right]}

\newcommand{\bp}{\boldsymbol{\rho}}
\newcommand{\bgam}{\boldsymbol{\gamma}}
\newcommand{\balpha}{\boldsymbol{\alpha}}

\newcommand{\beps}{\boldsymbol{\varepsilon}}

\newcommand{\cZ}{{\mathcal {Z}}}
\newcommand{\cX}{{\mathcal {X}}}
\newcommand{\bH}{{\mathbb H}}
\newcommand{\C}{{\mathbb C}}
\newcommand{\D}{{\mathbb D}}
\newcommand{\R}{{\mathbb R}}

\newcommand{\bl}{\boldsymbol{e_\ell}}
\newcommand{\br}{\boldsymbol{e_r}}   
\newcommand{\B}{{\bf B}}

\begin{document}

\author[V. Bolotnikov]{Vladimir Bolotnikov}
\address{Department of Mathematics, the College of William and Mary, Williamsburg, VA 23187-8795, USA}
\email{vladi@math.wm.edu}
\title[Finite Blaschke products over quaternions]{Finite Blaschke products over quaternions: unitary realizations and zero structure}

\maketitle

\begin{abstract}
We consider power series over the skew field $\mathbb H$ of real quaternions which 
are analogous to finite Blaschke products in the classical complex setting.
Several intrinsic characteriztions of such series are given in terms of their coefficients
as well as in terms of their left and right values. We also discuss the zero structure of finite Blaschke products 
including left/right zeros and their various multiplicities. We show how to construct a finite Blaschke product 
with prescribed zero structure. In particular, given a quaternion polynomial $p$ with all zeros less then one in modulus, 
we explicitly construct a power series $R$ with quaternion coefficients with no zeros such that $pR$ is a finite Blaschke product.
\end{abstract}

\section{Introduction}
\setcounter{equation}{0}
Given a complex polynomial $p(z)=p_0+p_1z+\ldots +p_nz^n$ with all zeros $a_1,\ldots, a_n$ in the open unit disk, 
the rational function
\begin{equation}
f(z)=\frac{p(z)}{z^np(1/\overline z)}=\frac{p_0+p_1z+\ldots +p_nz^n}
{\overline{p}_n+\overline{p}_{n-1}z+\ldots+\overline{p}_0z^n}
\label{1}
\end{equation}
can be written in the form 
\begin{equation}
f(z)=c\cdot \prod_{i=1}^n\frac{z-a_i}{1-z\overline{a}_i},\qquad |c|=1, \; |a_i|<1,
\label{2}
\end{equation}
and is called a (finite) Blaschke product of degree $n$.
Thus, to construct a finite Blaschke product with prescribed zeros as a given polynomial $p$), we just may use formulas 
\eqref{2} and \eqref{1}. The situation is quite different in the case of Blaschke products over quaternions. 

\smallskip

Finite Blaschke products over quaternions were introduced and studied in \cite{acs} within the theory
of slice-regular functions as certain $\star$-products of slice regular quaternionic Blaschke factors. 
These factors can be identified with formal power series 
$$
{\bf b}_{\alpha}(z):=(z-\alpha)(1-z\overline{\alpha})^{-1}=-\alpha+\sum_{k=0}^\infty (1-|\alpha|^2)\overline{\alpha}^kz^{k+1}
$$
where $\alpha\in\mathbb H$ is a fixed quaternion with $|\alpha|<1$, and $z$ is a formal variable commuting
with quaternionic coefficients. The latter power series can be evaluated on the left and on the right at 
any $\gamma\in\mathbb H$ with $|\gamma|\le 1$, giving rise to left- and right-regular automorphisms
of the closed unit ball of $\mathbb H$. The standard power-series product of finitely many Blaschke factors 
\begin{equation}
B={\bf b}_{\alpha_1}\cdots {\bf b}_{\alpha_n}
\label{m}
\end{equation}
is a formal power series over $\mathbb H$ that also can be evaluated on the left and on the right at
any quaternionic point $\gamma$ with $|\gamma|\le 1$ and produces the left and right slice-regular finite Blaschke
products in the sense of \cite{acs}. 

\smallskip

Due to non-commutativity of multiplication, the factorization of a given finite Blaschke product $B$ is largely
non-unique. Besides, any concrete factorization \eqref{m} does not display much information about zeros of $B$.
As was pointed out in \cite{acs}, in complete analogy to the quaternionic polynomial case, a finite Blaschke product
of the form \eqref{m} may have infinitely many zeros as well as only one left and one right zero. The detailed analysis 
of the zero structure of quaternionic finite Blaschke products is presented in Section 3.2. below.

\smallskip

Although the representation \eqref{m} is not unique, one can always factor out (on the left or on the right) 
the polynomial containing all information about the left/right zero structure of $B$.
Namely, $B$ can be (uniquely) factored as $B=pR$ where $p$ is a monic polynomial of degree $n$ and 
$R$ is a power series having no zeros. The inverse question (to recover $B$ from a given $p$) will be discussed in 
Section 4. We will also discuss several intrinsic characterization of finite Blaschke products as formal power series
(i.e., in terms of their coefficients) or as slice-regular functions (i.e., in terms of their left and right values).

\smallskip

The outline of the paper is as follows. In Section 2 we recall basic facts on power series over quaternions, 
their left and right evaluations and zero structure in terms of spherical divisors. In Section 3, we specify 
spherical divisors to the case of finite Blaschke products and present two characterizations of finite Blaschke products (in terms of 
the boundary behavior and in terms of unitary realizations). Three other characterizations (as isometric multipliers of 
the Hardy space and in terms of coefficients) are given in Section 5. In Section 4 we present several constructions of a Blaschke product
with prescribed zero structure.

\section{Preliminaries}
In what follows, $\mathbb H$ denotes the skew  field of quaternions  
\begin{equation}
\alpha=x_0+{\bf i}x_1+{\bf j}x_2+{\bf k}x_3 \qquad (x_0,x_1,x_2,x_3\in\mathbb R),
\label{1.3}   
\end{equation}
with imaginary units ${\bf i}, {\bf j}, {\bf k}$ subject to equalities
${\bf i}^2={\bf j}^2={\bf k}^2={\bf ijk}=-1$ and commuting with reals. The real and imaginary parts of $\alpha$ in \eqref{1.3} are
$\Re \alpha=x_0$ and  $\Im \alpha ={\bf i}x_1+{\bf j} x_2+{\bf k}x_3$, respectively.
The conjugate of $\alpha$ and its modulus are defined as
$$
\overline{\alpha}=\Re \alpha-\Im \alpha,\quad
|\alpha|=\sqrt{\alpha\overline{\alpha}}=\sqrt{(\Re\alpha)^2+|\Im\alpha|^2}
=\sqrt{x_0^2+x_1^2+x_2^2+x_3^2}.
$$
With a non-real $\alpha\in\mathbb H$, we associate its centralizer 
\begin{equation}
\C_\alpha:=\{\beta\in\mathbb H: \, \alpha\beta=\beta\alpha\}={\rm span}_{\mathbb R}({\bf 1}, \, \alpha)
={\rm span}_{\mathbb R}\big({\bf 1}, \, \frac{\Im \alpha}{|\Im\alpha|}\big)
\label{6}
\end{equation}
and the similarity (conjugacy) class
\begin{equation}
[\alpha]:=\{h\alpha h^{-1}: \; h\in\mathbb H\backslash\{0\}\}=
\{\beta\in\mathbb H: \; \Re\alpha =\Re\beta \; \;
\mbox{and} \; \; |\alpha|=|\beta|\}.
\label{7}
\end{equation}
and observe that $\C_\alpha\cap[\alpha]=\{\alpha,\overline{\alpha}\}$.
The second equality in \eqref{6} characterizes $\C_\alpha$
as the two-dimensional real subspace of $\mathbb H$ spanned by ${\bf 1}$ and $\alpha$
(or by ${\bf 1}$ and the purely imaginary unit
$\frac{\Im \alpha}{|\Im\alpha|}$), while the second equality in \eqref{7} (established in \cite{brenner})
describes $[\alpha]$  as the $2$-sphere (of radius $|\Im\alpha|$ around $\Re\alpha$). 
It is worth noting that the orthogonal complement $\C_\alpha^\perp$ of $\C_\alpha$ (in the euclidean metric 
of $\mathbb H\cong \mathbb R^4$) is the two-dimensional real subspace that also can be characterized 
(see e.g., \cite[Lemma 4.3]{bol1}) as the set of all intertwiners of $\alpha$ and $\overline{\alpha}$:
\begin{equation}
\C_\alpha^\perp=\{\beta\in\mathbb H: \, \alpha\beta=\beta\overline{\alpha}\}.
\label{8a}
\end{equation}
Since $\R$ is the center of $\mathbb H$), we conclude from
characterizations \eqref{6} and \eqref{8a} that 
\begin{equation}
\beta\gamma=\gamma\overline{\beta}\in\C_\alpha^\perp\quad\mbox{for all}\quad \gamma\in\C_\alpha, \; \beta\in\C_\alpha^\perp.
\label{jan7}
\end{equation}
\begin{Rk}
Given any fixed non-real $\alpha$ and any fixed unit $\beps\in\C_{\alpha}^\perp$,
every quaternion $\beta$ can be uniquely represented as
\begin{equation}
\beta=\beta_1+\beta_2\beps\quad\mbox{with}\quad \beta_1,\beta_2\in\C_\alpha.
\label{jan4}
\end{equation}
The representation \eqref{jan4} is orthogonal in the euclidean metric of $\mathbb H\cong \mathbb R^4$.
\label{R:1.1}
\end{Rk}
\begin{proof}
Indeed, it follows from characterizations \eqref{6} and \eqref{8a} that multiplication by $\beps$ (on either side) is an isometry
from $\C_\alpha$ onto $\C_\alpha^\perp$. Thus, $\beta_1$ and $\beta_2\beps$ in \eqref{jan4} are just orthogonal projections
of $\beta$ onto $\C_\alpha$ and $\C_{\alpha}^\perp$, respectively.\qed
\end{proof}

\noindent
{\bf 2.1. Power series and point evaluations}. We let  $\mathbb H[[z]]$ be the ring of 
formal power series $f(z)=\sum_{k\ge 0}z^kf_k$
in one variable $z$ which commutes with quaternionic coefficients $f_k\in\mathbb H$.
The sum and the product of two power series are defined as
\begin{equation}
(f+g)(z)=\sum_{k=0}^\infty z^k(f_k+g_k)\quad\mbox{and}\quad
(fg)(z)=\sum_{k=0}^\infty z^k\bigg(\sum_{\ell=0}^k f_\ell g_{k-\ell}\bigg).
\label{8}
\end{equation}
Quaternionic conjugation $\alpha\mapsto \overline{\alpha}$ on $\bH$ extends to the anti-linear 
involution $f\mapsto f^\sharp$ on $\bH[[z]]$ by the formula
\begin{equation}
f^\sharp(z)=\bigg(\sum_{k=0}^\infty z^kf_k\bigg)^\sharp=\sum_{k=0}^\infty z^k \overline{f}_k,
\label{9}
\end{equation}
and it is readily seen that for all $f,g\in\bH[[z]]$,
\begin{equation}
ff^\sharp=f^\sharp f\in\R[[z]],\quad (fg)^\sharp=g^\sharp f^\sharp,\quad (fg)(fg)^\sharp=(ff^\sharp)(gg^\sharp).
\label{2.10}
\end{equation}
To interpret power series as analytic functions, we need point evaluations of power series (at least within a ball 
$\mathbb B_R=\{\alpha\in\mathbb H: \, |\alpha|<R\}$. To this end, we restrict ourselves to the ring
 $$ 
\mathcal H_R=\big\{f(z)=\sum_{j=0}^\infty f_jz^j: \; \limsup_{k\to\infty} \sqrt[k]{|f_k|}\le 1/R\big\}.
 $$ 
For any $f\in\mathcal H_R$ and $\alpha\in \mathbb B_R$, we may define $f^{\bl}(\alpha)$ and $f^{\br}(\alpha)$ 
(left and right evaluations of $f$ at $\alpha$) by
 \begin{equation} 
f^{\bl}(\alpha)=\sum_{k=0}^\infty\alpha^k f_k,\quad f^{\br}(\alpha)=\sum_{k=0}^\infty 
f_k\alpha^k,\quad\mbox{if}\quad f(z)=\sum_{k=0}^\infty z^k f_k;
\label{15} 
\end{equation} 
The latter formulas make sense, since the condition 
$\;{\displaystyle\limsup_{k\to\infty}}\sqrt[k]{|f_k|}\le 1/R$ imposed on the coefficients guarantees
 the absolute convergence of the series in \eqref{15} for all $\alpha\in\mathbb B_R$. Since the 
multiplication in $\mathbb H$ is not commutative, left and right evaluations produce different results; however, 
equality 
\begin{equation} f^{\br}(\alpha)=\overline{f^{\sharp\bl}(\overline{\alpha})} 
\label{19} 
\end{equation}
holds for any $\alpha\in\mathbb H$ as is readily seen from \eqref{9} and \eqref{15} and enables us to translate
results concerning left evaluation to the right evaluation setting and vice versa. 

\smallskip

Let us note that if all the coefficients $f_k$ in \eqref{15} belong to the same plane $\C_\alpha$ (for some nonreal $\alpha$),
then left and right evaluations produce the same outcomes at each $\beta\in\C_\alpha$: $f^{\bl}(\beta)=f^{\br}(\beta)=f(\beta)\in\C_\alpha$.
In this particular case, the restriction of $f$ to the disk $\C_{\alpha}\bigcap P_R$ can be identified with a complex analytic function.
The general case reduces to the latter particular one upon making use of Remark \ref{R:1.1} as follows.
\begin{lemma}
Given an $f\in\mathcal H_R$, for any non-real $\alpha\in\mathbb B_R$ and any unit $\beps\in\C_\alpha^\perp$,
there exist two power series $g,h\in\C_{\alpha}[[z]]\cap \mathbb B_R$ such that $f=g+h\beps$. Furthermore, 
\begin{equation}
f^{\bl}(\beta)=g(\beta)+h(\beta)\beps, \quad f^{\br}(\beta)=g(\beta)+h(\overline{\beta})\beps\quad\mbox{for all} \; \; \beta\in\C_\alpha.
\label{jan5}
\end{equation}
\label{L:j1}
\end{lemma}
\begin{proof}
We represent each coefficient $f_k$ of $f(z)=\sum_{k\ge 0}f_kz^k$ as in \eqref{jan4}:
\begin{equation}
f_k=g_k+h_k\beps\quad\mbox{with}\quad g_k,h_k\in\C_\alpha\quad\mbox{for all}\; \; k\ge 0.
\label{jan9}
\end{equation}
Then the power series $g=\sum g_kz^k$ and $h=\sum h_kz^k$ provide the desired representation:
they belong to $\C_\alpha[[z]]$ (by construction); on the other hand, since the representations \eqref{jan9} are orthogonal, we have 
$$
|f_k|=\sqrt{|g_k|^2+|h_k\beps|^2}=\sqrt{|g_k|^2+|h_k|^2}\ge \max\{|g_k|,|h_k|\}
$$
and consequently, the series $g$ and $h$ belong to $\mathcal H_R$, since $f$ does.\qed
\end{proof}
The latter statement (which appeared in \cite{genstr} in the setting of slice-regular functions) provides an efficient way to extend 
certain results from complex function theory to various hypercomplex settings. One such extension 
(the quaternionic maximum modulus principle which also was pointed out in \cite{genstr}) is recalled below.
\begin{lemma}
Let $f\in\mathcal H_R$ and $\alpha\in\mathbb R$ be given. If $|f^{\bl}(\alpha)|\ge |f^{\bl}(\gamma)|$
(or if $|f^{\br}(\alpha)|\ge |f^{\br}(\gamma)|$) for all $\gamma$ in a neighborhood $\Delta$ of $\alpha$, then $f=f_0$. 
\label{L:im3}
\end{lemma}
\begin{proof}
Upon multiplying $f$ by a suitable unimodular constant on the right, we may assume without loss of generality that
$f^{\bl}(\alpha)>0$. If $\alpha_0$ is not real, consider the orthogonal representation $f=g+h\beps$ from Lemma \ref{L:j1}.
Since $f^{\bl}(\alpha)>0$, we have $h(\alpha)=0$. Then for any $\beta\in\C_\alpha\bigcap \Delta$, we have
\begin{equation}
|g(\alpha)|^2=|f^{\bl}(\alpha)|^2\ge |f^{\bl}(\beta)|^2=|g(\beta)|^2+|h(\beta)|^2\ge |g(\beta)|^2
\label{jan10}
\end{equation}
and thus, the maximum modulus principle for complex analytic functions applies to the restriction of $g$ to $\C_\alpha\bigcap \mathbb B_R$ 
forcing $g(\beta)=const=f^{\bl}(\alpha)$ for all $\beta\in\C_\alpha\bigcap \mathbb B_R$. Then it follows again from \eqref{jan10}
that $h(\beta)=0$ for all $\beta\in\C_\alpha\bigcap \mathbb B_R$ which in view of the first formula in \eqref{jan5} implies 
$f^{\bl}(\beta)$ does not depend on $\beta\in\C_\alpha\bigcap \mathbb B_R$ and therefore, $f^{\bl}(\beta)=f^{\bl}(0)=f_0$. The right version 
of the statement is justified similarly.\qed
\end{proof}
We finally recall that evaluation functionals $\bl$ and $\br$ are not multiplicative, in general. 
As is readily seen from \eqref{8} and \eqref{15}, for power series $f(z)$ and $g(z)$ as in \eqref{15}, 
\begin{equation}
(fg)^{\bl}(\alpha)=\sum_{k=0}^\infty\alpha^kf^{\bl}(\alpha)g_k\quad\mbox{and}\quad
 (fg)^{\br}(\alpha)=\sum_{k=0}^\infty f_kg^{\br}(\alpha)\alpha^k,
\label{16}
\end{equation}
from which it follows that 
\begin{align} 
(fg)^{\bl}(\alpha)
&=\left\{\begin{array}{ccc}
f^{\bl}(\alpha)\cdot g^{\bl}\left(f^{\bl}(\alpha)^{-1}\alpha
f^{\bl}(\alpha)\right)&\mbox{if} &
f^{\bl}(\alpha)\neq 0, \\
0 & \mbox{if} & f^{\bl}(\alpha)= 0,\end{array}\right.\label{16a}\\
(fg)^{\br}(\alpha)&=\left\{\begin{array}{ccc}
f^{\br}\left(g^{\br}(\alpha)\alpha
g^{\br}(\alpha)^{-1}\right)\cdot g^{\br}(\alpha)&\mbox{if} &
g^{\br}(\alpha)\neq 0, \\
0 & \mbox{if} & g^{\br}(\alpha)= 0.\end{array}\right.\label{16b}
\end{align}
\begin{Rk}
{\rm Since $x\beta=\beta x$ for all $x\in\R$ and $\beta\in\mathbb H$,
one can see from \eqref{15} that $f^{\bl}(x)=f^{\br}(x):=f(x)$ for every $x\in\R$. 
Besides, for every $x\in\R$, we have $f(x)^{-1}x f^{\bl}(x)=x=f^{\bl}(x)xf(x)^{-1}$ which on account of \eqref{16a}, \eqref{16b}
tell us that evaluation functionals at real points {\em are multiplicative}.}
\label{R:im7}
\end{Rk}
{\bf 2.2. Left and right zeros}. An element $\alpha\in \mathbb B_R$ is called a {\em left zero} of 
$f\in \mathcal H_R$ if $f^{\bl}(\alpha)=0$ or equivalently, if $f$ can be factored as 
$$
f=\bp_\alpha g\quad\mbox{for}\quad \bp_{\alpha}(z)=z-\alpha\quad\mbox{and some}\quad g\in\mathcal H_R.
$$
Right zeros are defined similarly. Since $f^{\bl}(x)=f^{\br}(x)$ for every $x\in\R$, all real zeros of $f$ are simultaneously 
left and right zeros. On the other hand, if $f\in\R[[z]]$, then 
$f^{\bl}(\alpha)=f^{\br}(\alpha)$ for every $\alpha\in\bH$, and hence, for power series $f$ with real coefficients,
the sets of left and right zeros coincide. Also, if $f\in\R[[z]]$, then for each $\alpha\in\bH$ and $h\neq 0$,
we have $f(h^{-1}\alpha h)=h^{-1}f(\alpha)h$ so that the zero set of $f$ contains, along with each $\alpha$, 
the whole similarity class $[\alpha]$. For example, given any non-real $\alpha$, the zero set of the real polynomial
\begin{equation}
\cX_{[\alpha]}(z)=(z-\alpha)(z-\overline{\alpha})=z^2-2z\cdot \Re(\alpha)+|\alpha|^2
\label{2.5}
\end{equation}
is the whole sphere $[\alpha]$. Observe that by the characterization \eqref{7},
the element $\alpha$ in \eqref{2.5} can be replaced by any other element in $[\alpha]$.
The polynomial \eqref{2.5} turns out to be equal to the least left (and right) common multiple
{\bf llcm} ({\bf lrcm}) of the polynomials $\bp_\beta$ and $\bp_\gamma$ for any two distinct elements $\beta,\gamma\in[\alpha]$.
Therefore, if $f\in\mathcal H_R$ has two distinct left (or right) zeros within the same similarity class
$V\subset \mathbb B_R$, then it can be factored  as $f=\cX_{V}g=g\cX_{V}$ for some $g\in\mathcal H_R$ and hence,
every element in $V$ turns out to be left and right zero of $f$. In this case, we say that $V$ is a 
{\em spherical zero} of $f$.

\smallskip

We next observe that given $f\in\mathcal H_R$, the zero set of the power series $ff^\sharp=f^\sharp f\in \R[[z]]$ is the union
of at most countably many similarity classes: $\cZ(ff^\sharp)=\bigcup_kV_k$, and it follows from evaluation formulas
\eqref{16a}, \eqref{16b} that all left and all right zeros of $f$ are contained in this union. If $V_k$ is not a spherical zero of $f$,
it contains at most one left and one right zeros of $f$. A remarkable fact discovered by Niven \cite{niven} is that 
it does contain them. Indeed if $V_k=\{x\}\subset\R$ (i.e., if 
$x$ is a real zero of $ff^\sharp$, then $x$ also is a real zero of $f$. The nonreal case is covered below.
\begin{theorem}
Given $f\in\mathcal H_R$, let $\alpha$ and $\overline{\alpha}$ be complex roots
(or any quaternion-conjugate roots) of the real power series $ff^\sharp$.
\begin{enumerate}
\item If $f^{\bl}(\alpha)=0\neq f^{\br}(\overline{\alpha})$, then the only right zero of $f$ in $[\alpha]$
is $(f^{\bl}(\overline{\alpha}))^{-1}\alpha f^{\bl}(\overline{\alpha})$.
\item If $f^{\bl}(\alpha)\neq 0$, then the only left root $\gamma_\ell$ and the only right root $\gamma_r$ of
$f$ in the similarity class $[\alpha]$ are given by
\begin{equation}
\begin{array}{l}
\gamma_\ell=(\overline{\alpha}f^{\bl}(\alpha)+\alpha f^{\bl}(\overline{\alpha}))(f^{\bl}(\alpha)+
f^{\bl}(\overline{\alpha}))^{-1},\\
\gamma_r=(f^{\bl}(\alpha)-f^{\bl}(\overline{\alpha}))^{-1}(\overline{\alpha}f^{\bl}(\alpha)-\alpha
f^{\bl}(\overline{\alpha})).\end{array}
\label{3.1}
\end{equation}
\end{enumerate}
\label{T:2.1}
\end{theorem}
Explicit formulas \eqref{3.1} above were obtained in \cite[Section 3]{bolapp}. In the polynomial setting, the 
idea to represent left/right zeros of $p\in\mathbb H[z]$ by means of complex roots of the real polynomial
$pp^\sharp$ goes back to Niven \cite{niven} and was used there to show 
any nonconstant polynomial over $\mathbb H$ has left and right zeros 
and consequently, that any (monic) polynomial $p\in\mathbb H[z]$ can be factored as
\begin{equation}
p(z)=(z-\alpha_1)(z-\alpha_2)\cdots(z-\alpha_n),\quad \alpha_1,\ldots,\alpha_n\in\bH.
\label{1.10}
\end{equation}
Following the terminology of \cite{ore}, let us say that a (monic) polynomial $p\in\mathbb H[z]$ 
is {\em indecomposable} if it cannot be represented as the {\bf lrcm} 
of its proper left divisors (equivalently, $p$  cannot be represented as the {\bf llcm} %least left common multiple 
of its proper right divisors). Indecomposable polynomials can be conveniently characterized in terms of factorizations \eqref{1.10}.
To this end, let us say that a finite ordered collection ${\balpha}=(\alpha_1,\ldots,\alpha_n)\subset\bH$ is a {\em
spherical chain} if all its elements belong to the same  similarity class (sphere) and no two
consecutive elements are quaternion-conjugates of each other:
\begin{equation}
\alpha_1\sim \alpha_2\sim\ldots\sim\alpha_n\quad\mbox{and}\quad \alpha_{j+1}\neq
\overline{\alpha}_j\quad\mbox{for}\quad j=1,\ldots,n-1.
\label{2.6}
\end{equation}
\begin{theorem}
Let $p\in\bH[z]$ be a monic polynomial factored as in \eqref{1.10}. 
The following are equivalent:
\begin{enumerate}
\item $p$ is indecomposable.
\item $\balpha=(\alpha_1,\ldots,\alpha_n)$ is a spherical chain.
\item $\alpha_1$ is the only left zero of $p$.
\item $\alpha_n$ is the only right zero of $p$.
\item \eqref{1.10} is a unique factorization of $f$ into the product of linear factors.
\end{enumerate}
\label{T:5.1}
\end{theorem}
The proof of Theorem \ref{T:5.1} can be found in \cite[Section 4]{bolapp}. Here we will demonstrate the 
implication $(2)\Rightarrow (3)$: for $p$ of the form \eqref{1.10}, we have $pp^\sharp=\cX_{[\alpha_1]}\cdots \cX_{[\alpha_n]}=
\cX_{[\alpha_1]}^n$, since all nodes $\alpha_k$ belong to the same similarity class. Therefore,
$p$ has no zeros outside $[\alpha_1]$, and since for any $\gamma\in[\alpha_1]$,
$$
f^{\bl}(\gamma)=(\gamma-\alpha_{1})(\overline{\alpha}_1-\alpha_{2})(\overline{\alpha}_2-\alpha_3)\cdots
(\overline{\alpha}_{n-1}-\alpha_{n}),
$$
it follows that $\alpha_1$ is the only left zero of $p$. 

\medskip
\noindent
{\bf 2.3. Zero structure}. The notion of zero multiplicity for quaternion power series (in particular, for
quaternion polynomials) is a delicate issue (for related results based on various notions of zero multiplicities, 
we refer to \cite{ps,gs,gs1,bolapp}). 

\smallskip

For $f\in\mathcal H_R$, we say that $x\in(-R,R)$ is zero of $f$ of multiplicity $\pi$ if $f=\bp_x^\pi g$ for some 
$g\in\mathcal H_R$ having no zero at $x$. Also, a similarity class $V\subset \mathbb B_R$ is said to be a spherical zero 
of $f$ of multiplicity $m_s(V,f)=k$ if $f=\cX_V^k g$ for some $g\in\mathcal H_R$ for which $V$ is not a spherical zero.
Unless $f=0$, both multiplicities must be finite as otherwise the real power series $ff^\sharp$ considered over the
complex disc $\mathbb B_R\bigcap \C$ would have zero of infinite multiplicity at $x$ or, respectively, at 
$\alpha\in V\bigcap \C$. The next result (see \cite[Section 3]{bolapp}) describes the zero structure of 
$f$ within a given similarity class.
\begin{theorem}
Given $f\in\mathcal H_R$, let $V$ be a similarity class containing zeros of $f$ and let 
$m_s(V,ff^\sharp)=k$.
Then there exist unique (monic) polynomials $D^f_{{\boldsymbol \ell}, V}$ and $D^f_{{\boldsymbol r}, V}$ with all 
zeros in $V$ such that
\begin{equation}
f=D^f_{{\boldsymbol \ell}, V}\cdot h_V=g_V\cdot D^f_{{\boldsymbol r}, V}\label{1.7}
\end{equation}
for some $h_V,g_V\in\mathcal H_R$ having no zeros in $V$. More precisely:
\begin{enumerate}
\item If $x$ is a real zero of $F$ of multiplicity $\pi$, then $D^f_{{\boldsymbol \ell}, \{x\}}=D^f_{{\boldsymbol r},
\{x\}}=\bp_x^{\pi}$.
\item If $m_s(V;ff^\sharp)=2m_s(V;f)=2\kappa>0$, then $D^f_{{\boldsymbol \ell}, V}=D^f_{{\boldsymbol r},
V}=\cX_{V}^{\kappa}$.
\item If $m_s(V;f)=\kappa\ge 0$ and $2m_s(V;f)-m_s(V;ff^\sharp)=k>0$, then there exist unique spherical chains
$\balpha=(\alpha_{1},\ldots,\alpha_{k})$ and
$\widetilde\balpha=(\widetilde\alpha_{1},\ldots,\widetilde\alpha_{k})$ in $V$ such that
\begin{equation}
D^f_{{\boldsymbol \ell}, V}=\cX_{V}^{\kappa}\cdot\bp_{\alpha_{1}}\cdots \bp_{\alpha_{k}}
\quad\mbox{and}\quad D^f_{{\boldsymbol r}, V}= \bp_{\widetilde\alpha_{k}}\cdots
\bp_{\widetilde\alpha_{1}}\cX_{V}^{\kappa}.
\label{4.15}
\end{equation}
\end{enumerate}
\label{T:1.1}
\end{theorem}
Following \cite{bolapp}, we will refer to $D^f_{{\boldsymbol \ell}, V}$ and $D^f_{{\boldsymbol r}, V}$ as to {\em left} 
and {\em right spherical divisors} of $f$. Left (right) spherical divisors corresponding to different spheres are left (right) coprime.
If the zero set of $f\in\mathcal H_R$ is contained in the union of finitely many spheres $V_1\cup\cdots\cup V_n$ and if 
$D^f_{{\boldsymbol \ell},V_k}$ and $D^f_{{\boldsymbol r}, V_k}$ are the corresponding spherical divisors, 
then $f$ can be factored as 
$$
f=P^f_\ell h=gP^f_r 
$$
where $P^f_\ell$ and $P^f_r$ are defined as the least common multiples 
$$
P^f_\ell={\bf lrcm}(D^f_{{\boldsymbol \ell}, V_1},\ldots,D^f_{{\boldsymbol \ell}, V_n}),\qquad
P^f_r={\bf llcm}(D^f_{{\boldsymbol r}, V_1},\ldots,D^f_{{\boldsymbol r}, V_n}),
$$
and $h,g\in \mathcal H_R$  have no zeros. In this case, $P^f_\ell$ and $P^f_r$ contain all information about the 
left (right) zero structure of $f$.

\smallskip

We conclude this section with an example of a power series having no zeros. Let us consider the power series
\begin{equation}
{\bf k}_{\alpha}(z)={\displaystyle\sum_{k=0}^\infty \overline{\alpha}^kz^k}\quad(\alpha\in\mathbb B_1)
\label{jan13}
\end{equation}
that clearly belongs to $\mathcal H_{|\alpha|^{-1}}$. Since the real power series
$$
{\bf k}_{\alpha}(z){\bf k}^\sharp_{\alpha}(z)=(1-z(\alpha+\overline{\alpha})+z^2|\alpha|^2)^{-1}\in\R[[z]]
$$
has no zeros, ${\bf k}_{\alpha}$ has no zeros either. 
Applying the left evaluation functional to the equality 
$(|\alpha|^2\cX_{[\alpha^{-1}]}{\bf k}_{\alpha})(z)=1-z\alpha$ we get
$$
(1-\gamma(\alpha+\overline{\alpha})+\gamma^2|\alpha|^2){\bf k}_{\alpha}^{\bl}(\gamma)=1-\gamma\alpha,
$$
which implies
\begin{equation}
{\bf k}_{\alpha}^{\bl}(\gamma)=(1-\gamma(\alpha+\overline{\alpha})+\gamma^2|\alpha|^2)^{-1}(1-\gamma\alpha).
\label{20}
\end{equation}
Similarly, evaluating the equality $|\alpha|^2{\bf k}_{\alpha}\cX_{[\alpha^{-1}]}=1-z\alpha$ on the right at 
$\gamma$ gives
\begin{equation}
{\bf k}_{\alpha}^{\br}(\gamma)=(1-\alpha\gamma)(1-\gamma(\alpha+\overline{\alpha})+\gamma^2|\alpha|^2)^{-1}.
\label{20a}
\end{equation}

\section{Finite Blaschke products as Schur-class power series}
\setcounter{equation}{0}

The classical Schur class $\mathcal S$ consists of complex functions that are analytic 
on the open unit disk $\D\subset\C$ and map $\D$ into its closure.
By the maximum modulus principle, each Schur function is either an analytic self-mapping
of $\D$ or a unimodular constant. In the seminal paper \cite{Schur}, Schur-class functions
were characterized as power series such that the lower triangular Toeplitz matrix constructed 
from its coefficients by the formula
\begin{equation}
{\bf T}_n^f=\left[\begin{array}{cccc}f_{0} & 0 & \ldots & 0
\\ f_{1}& f_{0} & \ddots & \vdots \\ \vdots& \ddots & \ddots & 0
\\ f_{n-1}&  \ldots & f_{1} & f_{0}\end{array}\right],\quad\mbox{where}\quad f(z)=\sum_{k=0}^\infty f_kz^k,
\label{1.4}
\end{equation}
is contractive for all $n\ge 1$. For a fixed $n$, the later property is equivalent to the inequality
$\|fp\|_{H^2(\D)}\le \|p\|_{H^2(\D)}$ holding for any polynomial $p\in\C[z]$ with $\deg (p)\le n$ (where $H^2(\D)$ stands for 
the Hardy space of power series with square summable coefficients). 
Since polynomials are dense in $H^2(\D)$, the functions $f\in\mathcal S$ are characterized by the 
above norm inequality holding for all $p\in H^2(\D)$; in other words, the operator $M_f: \, p\to fp$ of multiplication 
by $f$ is a contraction on  $H^2(\D)$. Thus, the Schur class coincides with the class of {\em contractive multipliers} 
of the Hardy space $H^2(\D)$. 

\smallskip
\noindent
{\bf 3.1. The Schur class $\mathcal S_{\mathbb H}$} of quaternionic power series can be introduced in  at least five different ways:
in terms of left and right values, in terms of (quaternionic) matrices \eqref{3.1}, or as (left or right) contractive multipliers of 
the space
\begin{equation}
H^2(\mathbb B_1)=\bigg\{f(z)=\sum_{k=0}^\infty f_kz^k\in \mathbb H[[z]]: \; \|f\|^2=\sum_{k=0}^\infty |f_k|^2<\infty\bigg\},
\label{dop4}
\end{equation}
the quaternionic counter-part of $H^2(\D)$ (note that the square-summability condition guarantees the inclusion 
$H^2(\mathbb B_1)\subset \mathcal H_1$). The next theorem (see \cite{abcs} for the proof)
shows that all five ways are equivalent. Power series $f\in \mathbb H[[z]]$ satisfying one (and therefore all)
of five conditions below will be called a {\em Schur power series}.
\begin{theorem}
Let $f\in \bH[[z]]$ be as in \eqref{1.4}. The following are equivalent:
\begin{enumerate}
\item $|f^{\bl}(\alpha)|\le 1$ for all $\alpha\in\mathbb B_1$.
\item $|f^{\br}(\alpha)|\le 1$ for all $\alpha\in\mathbb B_1$.
\item The matrix ${\bf T}_n^f$ is contractive for all $n\ge 1$.
\item $\|fh\|_{H^2(\mathbb B_1)}\le \|h\|_{H^2(\mathbb B_1)}$ for all $h\in H^2(\mathbb B_1)$.
\item $\|hf\|_{H^2(\mathbb B_1)}\le \|h\|_{H^2(\mathbb B_1)}$ for all $h\in H^2(\mathbb B_1)$.
\end{enumerate}
\label{T:3.1}
\end{theorem}
\begin{Rk}
{\rm If $f$ is not a unimodular constant, then the strict inequalities hold in pats (1) and (2) above, 
by the maximum modulus principle in Lemma \ref{L:im3}.}
\label{R:3.1}
\end{Rk}
\begin{Rk}
{\em The equivalence $(4)\Leftrightarrow(5)$ in Theorem \ref{T:3.1} is not immediate. For example, for  
$f=\frac{1}{2}+\frac{\bf k}{2}z\in\mathcal S_{\mathbb H}$ and $h={\bf i}+{\bf j}z$, we have
$fh=\frac{1}{2}({\bf i}+2{\bf j}z-{\bf i}z^2)$ and $hf=\frac{\bf i}{2}(1+z^2)$ so that 
$\|fh\|^2_{H^2(\mathbb B_1)}=\frac{3}{2}\neq \frac{1}{2}=\|hf\|^2_{H^2(\mathbb B_1)}$.
However, due to \eqref{19}, the equivalence $(4)\Leftrightarrow(5)$ means that $f\in\mathcal S_{\mathbb H}$
if and only if $f^\sharp\mathcal S_{\mathbb H}$. Since $\|h\|_{H^2(\mathbb B_1)}=\|h^\sharp\|_{H^2(\mathbb B_1)}$,
by definition \eqref{dop4}), then we can derive (5) from (4) as follows:
$$
\|hf\|_{H^2(\mathbb B_1)}=\|f^\sharp h^\sharp\|_{H^2(\mathbb B_1)}\le \|h^\sharp\|_{H^2(\mathbb B_1)}=\|h\|_{H^2(\mathbb B_1)}.
$$}
\label{R:3.2}
\end{Rk}
\noindent
{\bf 3.2. Blaschke factors}. We now recall the simplest non-constant examples of Schur power series, which 
has already draw some attention as slice-regular automorphisms of the quaternionic unit ball $\mathbb B_1$ 
\cite{bige,bist,heja} and isometric multipliers of the Hardy space $H^2(\mathbb B_1)$ \cite{acs,acsbook}.
Observe that the series \eqref{jan13} satisfies the identity 
${\bf k}_{\alpha}(z)\cdot (1-z\overline\alpha)\equiv 1$ meaning that ${\bf k}_{\alpha}$
is the formal inverse of the polynomial $1-z\overline\alpha$. Hence, the power series
\begin{align}
{\bf b}_{\alpha}(z)&:=\bp_{\alpha}(z)\cdot {\bf k}_{\alpha}(z)={\bf k}_{\alpha}(z)\cdot \bp_{\alpha}(z)\notag\\
&=(z-\alpha)\cdot \sum_{k=0}^\infty
\overline{\alpha}^kz^k=-\alpha+(1-|\alpha|^2)\cdot  \sum_{k=0}^\infty \overline{\alpha}^kz^{k+1}
\label{9.4}
\end{align}
can be viewed as the quaternionic analog of the Blaschke factor $\frac{z-\alpha}{1-z\overline{\alpha}}$.
It is clear from \eqref{9.4} that $\alpha$ is the only left (and right) zero of 
${\bf b}_\alpha$. Combining \eqref{9.4} and \eqref{20}, \eqref{20a} gives
\begin{align}
{\bf b}_{\alpha}^{\bl}(\gamma)&=\gamma {\bf k}_{\alpha}^{\bl}(\gamma)-{\bf k}_{\alpha}^{\bl}(\gamma)\alpha\notag\\
&=(1-\gamma(\alpha+\overline{\alpha})+\gamma^2|\alpha|^2)^{-1}\big(
\gamma(1-\gamma\alpha)-(1-\gamma\alpha)\alpha\big),\label{22}\\
{\bf b}_{\alpha}^{\br}(\gamma)&={\bf k}_{\alpha}^{\br}(\gamma)\gamma-\alpha{\bf k}_{\alpha}^{\br}(\gamma)\notag\\
&=\big((1-\alpha\gamma)\gamma-\alpha(1-\alpha\gamma)\big)
(1-\gamma(\alpha+\overline{\alpha})+\gamma^2|\alpha|^2)^{-1}.\label{21}
\end{align}
Making use of notation
\begin{equation}
\gamma_\ell:=(1-\gamma\alpha)^{-1}\gamma(1-\gamma\alpha),\qquad \gamma_r:=(1-\alpha\gamma)\gamma(1-\alpha\gamma)^{-1}
\label{23}
\end{equation}
and observing the equalities
$$
1-\gamma(\alpha+\overline{\alpha})+\gamma^2|\alpha|^2=(1-\gamma\alpha)(1-\gamma_\ell\overline{\alpha})=
(1-\overline{\alpha}\gamma_r)(1-\alpha\gamma),
$$
we may write the formulas \eqref{22}, \eqref{21} as 
\begin{equation}
{\bf b}_{\alpha}^{\bl}(\gamma)=(1-\gamma_\ell\overline{\alpha})^{-1}(\gamma_\ell-\alpha),\qquad
{\bf b}_{\alpha}^{\br}(\gamma)=(\gamma_r-\alpha)(1-\overline{\alpha}\gamma_r)^{-1}.
\label{20b}
\end{equation}
\begin{Rk}
{\rm Observe that in case $\gamma\in\C_\alpha$ (i.e., if $\gamma\alpha=\alpha\gamma$), the formulas \eqref{20b} 
amount to the usual point evaluation:
$$
{\bf b}_{\alpha}^{\bl}(\gamma)={\bf b}_{\alpha}^{\br}(\gamma)=(\gamma-\alpha)(1-\gamma\overline{\alpha})^{-1}\quad
\mbox{for all} \; \; \gamma\in\C_\alpha.
$$
On the other hand, if $\gamma\in[\alpha]$ (i.e., $\gamma\sim \alpha$), then $\gamma^2=\gamma(\alpha+\overline{\alpha})-|\alpha|^2$ and therefore,
$$
1-\gamma(\alpha+\overline{\alpha})+\gamma^2|\alpha|^2=(1-|\alpha|^2)(1-\gamma^2),
$$
$$
\gamma(1-\gamma\alpha)-(1-\gamma\alpha)\alpha=(1-|\alpha|^2)(\gamma-\alpha)=
(1-\alpha\gamma)\gamma-\alpha(1-\alpha\gamma.
$$
With these substitutions, we conclude from \eqref{22}, \eqref{21} that
\begin{equation}
{\bf b}_{\alpha}^{\bl}(\gamma)=(1-\gamma^2)^{-1}(\gamma-\alpha),\quad
{\bf b}_{\alpha}^{\br}(\gamma)=(\gamma-\alpha)(1-\gamma^2)^{-1} \quad\mbox{if}\; \; \gamma\in[\alpha].
\label{24}
\end{equation}}
\label{R:3.3}
\end{Rk}
\begin{proposition}
The Blaschke factor ${\bf b}_\alpha$ belongs to $\mathcal S_{\mathbb H}$. Moreover,
\begin{equation}
|{\bf b}_{\alpha}^{\bl}(\gamma)|=|{\bf b}_{\alpha}^{\br}(\gamma)| \; 
\begin{array}{ccc}<1, & \mbox{\; if} & |\gamma|<1, \\ =1, & \; \mbox{if} & |\gamma|=1.\end{array}
\label{20e}
\end{equation}
\label{P:91}
\end{proposition}
\begin{proof}
Making use of formulas \eqref{20b} we get the equalities
\begin{align}
1-|{\bf b}_{\alpha}^{\br}(\gamma)|^2&=(1-|\alpha|^2)(1-|\gamma|^2)|1-\gamma_\ell\overline\alpha|^{-2},\label{25}\\
1-|{\bf b}_{\alpha}^{\bl}(\gamma)|^2&=(1-|\alpha|^2)(1-|\gamma|^2)|1-\overline\alpha\gamma_r|^{-2}.\label{26}
\end{align}
which imply all the relations in \eqref{20e}, except for the equality 
$|{\bf b}_{\alpha}^{\bl}(\gamma)|=|{\bf b}_{\alpha}^{\br}(\gamma)|$ for $\gamma\in\mathbb B_1$. To derive
the latter equality from \eqref{25}, \eqref{26}, it suffices to verify that 
\begin{equation}
|1-\gamma_\ell\overline\alpha|=|1-\overline\alpha\gamma_r|.
\label{20d}
\end{equation}
To this end, we observe that since $(1-\gamma\alpha)^{-1}\gamma=\gamma(1-\alpha\gamma)^{-1}$ and $(1-\gamma\alpha)\overline{\alpha}=
\overline{\alpha}(1-\alpha\gamma)$, and since $\Re(uv)=\Re(vu)$ for all $u,v\in\mathbb H$, we have
\begin{align*}
\Re(\gamma_\ell\overline{\alpha})&=\Re((1-\gamma\alpha)^{-1}\gamma(1-\gamma\alpha)\overline{\alpha})\\
&=\Re(\gamma(1-\alpha\gamma)^{-1}\overline{\alpha}(1-\alpha\gamma))\\
&=\Re(\overline{\alpha}(1-\alpha\gamma)\gamma(1-\alpha\gamma)^{-1}=\Re(\overline{\alpha}\gamma_r).
\end{align*}
Since $\gamma_\ell$ and $\gamma_r$ belong to the similarity class $[\gamma]$, by definitions \eqref{23}, 
we have in particular, $|\gamma_\ell|=|\gamma_r|$, due to the characterization \eqref{7} of $[\gamma]$. Hence,
$$
|1-\gamma_\ell\overline\alpha|^2-|1-\overline\alpha\gamma_r|^2=2\Re(\overline\alpha\gamma_r)-
2\Re(\gamma_\ell\overline\alpha)+(|\gamma_\ell|^2-|\gamma_r|^2)|\alpha|^2=0,
$$
from which \eqref{20d} follows, thus completing the proof of \eqref{20e}. \qed
\end{proof}
Note that one can use formulas \eqref{20b} and \eqref{23} to recover the element $\gamma$ from its left 
or right image under ${\bf b}_{\alpha}$ as follows: from the first formula in \eqref{20b}, we have
$$
\gamma_\ell=(\alpha+{\bf b}_{\alpha}^{\bl}(\gamma))(1+\overline{\alpha}{\bf b}_{\alpha}^{\bl}(\gamma))^{-1}
=(1+{\bf b}_{\alpha}^{\bl}(\gamma)\overline{\alpha})^{-1}(\alpha+{\bf b}_{\alpha}^{\bl}(\gamma)),
$$
while from the first formula in \eqref{23} we get
$$
\gamma=(1-\alpha\overline{\gamma}_\ell)\gamma_\ell(1-\alpha\overline{\gamma}_\ell)^{-1}=
(1-\gamma_\ell\overline{\alpha})^{-1}\gamma_\ell(1-\gamma_\ell\overline{\alpha}).
$$
Combining the two latter formulas, one can see that actually, 
$\gamma={\bf b}_{-\alpha}^{\bl}({\bf b}_{\alpha}^{\bl}(\gamma))$. Similarly, it turns out that also
$\gamma={\bf b}_{-\alpha}^{\br}({\bf b}_{\alpha}^{\br}(\gamma))$.

\smallskip

Thus, in analogy to the complex case, ${\bf b}_{\alpha}$ gives rise to two automorphisms 
$\gamma\mapsto {\bf b}_{\alpha}^{\bl}(\gamma)$ and $\gamma\mapsto {\bf b}_{\alpha}^{\br}(\gamma)$ 
of the closed unit ball $\overline{\mathbb B}_1$.
\begin{Rk}
{\rm Given a Blaschke factor ${\bf b}_\alpha$ with a non-real zero $\alpha$, its conjugate (in the sense
of \eqref{9}) ${\bf b}_\alpha^\sharp={\bf b}_{\overline{\alpha}}$ turns out to be a Blaschke factor vanishing at $\overline{\alpha}$.
Their product
\begin{align}
\mathcal B_{[\alpha]}&:={\bf b}_\alpha{\bf b}_\alpha^\sharp=\bp_\alpha\bp_\alpha^\sharp{\bf k}_\alpha{\bf k}_\alpha^\sharp
=\cX_{[\alpha]}{\bf k}_\alpha{\bf k}_\alpha^\sharp\label{27}\\
&=(z^2-z(\alpha+\overline{\alpha})+|\alpha|^2)(1-z(\alpha+\overline{\alpha})+z^2|\alpha|^2)^{-1}\in\R[[z]]\notag
\end{align}
is a Blaschke product of degree two whose zero set coincides with the whole $[\alpha]$. Following the terminology of \cite{acs,acsbook},
we will call $\mathcal B_{[\alpha]}$ the Blaschke factor of the sphere $[\alpha]$.} 
\label{R:im4}
\end{Rk}

\medskip
\noindent
{\bf 3.3. Finite Blaschke products}. In analogy to the complex case \eqref{2}, a finite Blaschke product over $\mathbb H$
is defined as the power-series product of finitely many Blaschke factors of the form \eqref{9.4}:
\begin{equation}
\B(z)={\bf b}_{\alpha_1}(z){\bf b}_{\alpha_2}(z)\cdots {\bf b}_{\alpha_n}(z)\phi,\qquad
(\alpha_1,\ldots,\alpha_n\in\mathbb B_1, \;  |\phi|=1).
\label{9.6}
\end{equation}
It is clear that $\B$ belongs to $\mathcal H_R$ with $R=(\max\{|\alpha_1|,\ldots,|\alpha_n|\})^{-1}>1$. Therefore
evaluations $\B^{\bl}(\alpha)$ and $\B^{\br}(\alpha)$ make sense for any $\alpha\in\overline{\mathbb B}_1$. The integer 
$n$ in \eqref{9.6} is called the degree of $\B$; for the sake of uniformity, we will interpret unimodular constants as 
Blaschke products of degree zero.
\begin{Rk}
{\rm Placing the constant unimodular factor $\phi$ on the right is just a notational convention. Although multiplication in $\mathbb H$
is non-commutative, we still have  ${\bf b}_\alpha \cdot \phi=\phi \cdot {\bf b}_{\phi^{-1}\alpha\phi}$ for any $\alpha\in\mathbb B_1$
and $\phi\neq 0$. Thus, although we may equivalently define a finite Blaschke product as a product of Blaschke factors with extra 
unimodular constant factors on the left and even between the non-constant factors, still, at the expense of changing the nodes 
of the Blaschke factors, each such expression can be transformed to the form with the only unimodular factor on the right as in \eqref{9.6}. 
Now we can state that the product of two finite Blaschke products is again a finite Blaschke product.} 
\label{T:im5}
\end{Rk}
\begin{proposition}
For a finite Blaschke product \eqref{9.6}, $|\B^{\bl}(\gamma)|<1$ and $|\B^{\br}(\gamma)|<1$ for $|\gamma|<1$, and
$|\B^{\bl}(\gamma)|=|\B^{\br}(\gamma)|=1$ if $|\gamma|=1$.
\label{P:8.1}
\end{proposition}
\begin{proof}
If $|\gamma|=1$, then by recursively applying the formula \eqref{16a} to the product \eqref{9.6} we get
\begin{equation}
\B^{\bl}(\gamma)={\bf b}^{\bl}_{\alpha_1}(\gamma){\bf b}^{\bl}_{\alpha_2}(\gamma_1)\cdots {\bf b}^{\bl}_{\alpha_n}(\gamma_{n-1})\phi
\label{9.7}
\end{equation}
for suitable $\gamma_1,\ldots,\gamma_{n-1}\in[\gamma]$. By the characterization \eqref{7}, $|\gamma_k|=1$ and therefore,
$|{\bf b}^{\bl}_{\alpha_1}|=|{\bf b}^{\bl}_{\alpha_{k+1}}(\gamma_k)|=1$ for $k=1,\ldots,n-1$, by property \eqref{20e}. 
It is now seen from \eqref{9.7} that $|\B^{\bl}(\gamma)|=1$. If $|\gamma|<1$, then either ${\bf b}^{\bl}_{\alpha_{k+1}}(\gamma_k)=0$ for some $k$
and therefore, $\B^{\bl}(\gamma)=0$, or we will get the product \eqref{9.7} with
$\gamma_1,\ldots,\gamma_n\in[\gamma]$, this time concluding that $|\gamma_k|=|\gamma|<1$ and therefore,
$|{\bf b}^{\bl}_{\alpha_{k+1}}(\gamma_k)|<1$, again by property \eqref{20e}. Then we conclude from \eqref{9.7} that $|\B^{\bl}(\gamma)|<1$.
The statements concerning $\B^{\br}(\gamma)$ are verified similarly.\qed
\end{proof}
Note that Proposition \ref{P:8.1} does not assert the equality $|\B^{\bl}(\gamma)|=|\B^{\br}(\gamma)|$ for all $\gamma\in\mathbb B_1$.
As matter of fact, this equality may fail even for Blaschke products of degree two. For example, let $\alpha\sim\gamma$ be two noncommuting
similar elements (i.e., $\gamma\neq \alpha,\overline{\alpha}$) and let us consider the Blaschke product
$\B={\bf b}_{\alpha}{\bf b}_{\gamma}$. By formula \eqref{20b}, $\B^{\br}(\gamma)=0$, while the double application of the first formula
in \eqref{24} combined with \eqref{16a} leads us to
\begin{align*}
\B^{\bl}(\gamma)&={\bf b}^{\bl}_{\alpha}(\gamma)\big(1-{\bf b}^{\bl}_{\alpha}(\gamma)^{-1}\gamma^2{\bf b}^{\bl}_{\alpha}(\gamma)\big)^{-1}
\big({\bf b}^{\bl}_{\alpha}(\gamma)^{-1}\gamma{\bf b}^{\bl}_{\alpha}(\gamma)-\gamma\big)\\
&=(1-\gamma^2)^{-1}\big(\gamma{\bf b}^{\bl}_{\alpha}(\gamma)-{\bf b}^{\bl}_{\alpha}(\gamma)\gamma\big)\\
&=(1-\gamma^2)^{-2}\big(\alpha\gamma-\gamma\alpha)\neq 0.
\end{align*}
We next make several comments regarding the zero structure of the Blaschke product \eqref{9.6}.
Any real $\alpha_k\in\R$ in \eqref{9.6} is a two-sided zero of $\B$. If all the nodes 
in \eqref{9.6} are non-real, then we have, by \eqref{2.10} and \eqref{27}, 
$\B\B^\sharp=\prod_{k=1}^n \mathcal B_{[\alpha_k]}\in \R[[z]]$; hence all zeros of $\B$ are contained in 
the union of the (not necessarily distinct) similarity classes $[\alpha_1],\ldots,[\alpha_n]$ and can be found as 
suggested in Theorem \ref{T:2.1}. The zero structure of $\B$ within each similarity class is described in
Theorem \ref{T:1.1}. Within the Schur class $\mathcal S_{\mathbb H}$, it makes sense to split off
zeros of a power series using finite Blaschke products rather than polynomials. Details are given below.
\begin{theorem}
Given $f\in\mathcal S_{\mathbb H}$, let $V$ be a similarity class containing zeros of $f$ and let
$m_s(V,ff^\sharp)=k$.
Then there exist unique finite Blaschke products $B^f_{{\boldsymbol \ell}, V}$ and $B^f_{{\boldsymbol r}, V}$ with all
zeros in $V$ such that
\begin{equation}
f=B^f_{{\boldsymbol \ell}, V}\cdot H_V=G_V\cdot B^f_{{\boldsymbol r}, V}\label{1.7hyp}
\end{equation}
for some $H_V,G_V\in\mathcal S_{\mathbb H}$ having no zeros in $V$. More precisely:
\begin{enumerate}
\item If $x$ is a real zero of $F$ of multiplicity $\pi$, then $B^f_{{\boldsymbol \ell}, \{x\}}=B^f_{{\boldsymbol r},
\{x\}}={\bf b}_x^{\pi}$.
\item If $m_s(V;ff^\sharp)=2m_s(V;f)=2\kappa>0$, then $B^f_{{\boldsymbol \ell}, V}=B^f_{{\boldsymbol r},
V}=\mathcal B_{V}^{\kappa}$.
\item If $m_s(V;f)=\kappa\ge 0$ and $m_s(V;ff^\sharp)-2m_s(V;f)=k>0$, then there exist unique spherical chains
$\balpha=(\alpha_{1},\ldots,\alpha_{k})$ and
$\widetilde\balpha=(\widetilde\alpha_{1},\ldots,\widetilde\alpha_{k})$ in $V$ such that
\begin{equation}
B^f_{{\boldsymbol \ell}, V}=\mathcal B_{V}^{\kappa}\cdot{\bf b}_{\alpha_{1}}\cdots {\bf b}_{\alpha_{k}}
\quad\mbox{and}\quad B^f_{{\boldsymbol r}, V}= {\bf b}_{\widetilde\alpha_{k}}\cdots
{\bf b}_{\widetilde\alpha_{1}}\mathcal B_{V}^{\kappa}.
\label{4.15hyp}
\end{equation}
\end{enumerate}
\label{T:1.1h}
\end{theorem}
\begin{proof}
If we only assume that $f\in\mathcal H_1$, then all statements are true with some $H_V,G_V\in\mathcal H_1$ having no zeros in $V$.
Indeed, for case (1) we apply Theorem \ref{T:1.1} and get the desired factorization
$$
f=\bp_x^\pi h_V={\bf b}_x^\pi H_V \quad\mbox{with}\quad H_V(z)=(1-zx)^\pi h_V. 
$$
For case (2) we again use Theorem \ref{T:1.1} and pick any $\alpha\in V$ to get the factorization 
\begin{equation}
f=\cX_V^\kappa h_V=\mathcal B_{V}^{\kappa}H_V \quad\mbox{with}\quad H_V(z)=(1-z(\alpha+\overline{\alpha})+z^2|\alpha|^2)^\kappa h_V
\label{1.1hyp}
\end{equation}
as in \eqref{1.7hyp}. We will process the third case in two steps: we first factor $f$ as in \eqref{1.1hyp}, but this time we can only guarantee
that $V$ is not a spherical zero of $H_V$. However $m_s(V,H_VH_V^\sharp)=k>0$, and therefore, $H_V$ has a unique left zero $\alpha_1$ in $V$ 
which can be found as suggested in Theorem \ref{T:2.1}. The we carry out the following recursion: letting $H_1:=H_V$ we factor
$H_{j}={\bf b}_{\alpha_j}H_{j+1}$, where $\alpha_j$ is the only left zero of $H_j$ in $V$, for $j=1,\ldots,k$. For each $j$, the power series
$H_j$ belongs to $\mathcal H_1$ and has no spherical zero at $V$ (as $H_j$ is a factor of $H_V$). Furthermore, 
$$
m_s(V;H_jH_j^\sharp)=k-j+1,
$$
and hence $H_j$ has a unique left zero in $V$ for $j=1,\ldots,k$. Since $m_s(V;H_{k+1}H_{k+1}^\sharp)=0$, 
the power series $H_{k+1}$ has no zeros in $V$. As the result of the latter 
recursion we get the factorization 
\begin{equation}
f=\mathcal B_{V}^{\kappa}H_V=\mathcal B_{V}^{\kappa}{\bf b}_{\alpha_{1}}\cdots {\bf b}_{\alpha_{k}}H_{k+1}.
\label{1.2hyp}
\end{equation}
The elements $(\alpha_1,\ldots,\alpha_k)$ form a spherical chain in $V$ (see \eqref{2.6}), since the equality $\alpha_j=\overline{\alpha}_{j+1}$
would have produced the factor ${\bf b}_{\alpha_{j}}{\bf b}_{\alpha_{j}}={\bf b}_{\alpha_{j}}{\bf b}_{\alpha_{j}}^\sharp=\mathcal B_V$ in \eqref{1.2hyp}
which is impossible as $H_V$ has no spherical zero at $V$, by construction. Since $H_{k+1}$ has no zeros in $V$, the representation
\eqref{1.2hyp} justifies the first representation in \eqref{1.7hyp} for the case (3). The second representation with $B^f_{{\boldsymbol r}, V}$
of the form \eqref{4.15hyp} is justified similarly. 

\smallskip

A remarkable property of Schur-class power series is that if $f\in \mathcal S_{\mathbb H}$ has a left zero at $\alpha\in\mathbb B_1$, 
then  the power series $h$ appearing in the factorization $f={\bf b}_\alpha h$ also belongs to $\mathcal S_{\mathbb H}$. This is the 
(left) quaternionic version (not the most general though -- see \cite{bige,acs1}) of the classical Schwarz-Pick lemma. The right version 
follows from the left one upon making use of the relation \eqref{19}: {\em if $f\in \mathcal S_{\mathbb H}$  has a right zero 
at $\beta\in\mathbb B_1$, then $f=g{\bf b}_\beta$, where $g\in\mathcal S_{\mathbb H}$}. 

\smallskip

By recursively applying the left and the right Schwarz-Pick lemma to factorizations \eqref{1.7hyp} of $f\in \mathcal S_{\mathbb H}$, 
we conclude that $H_V$ and $G_V$ belong to $\mathcal S_{\mathbb H}$ which completes the proof of the theorem.  \qed
\end{proof}
In the next section we will get back to the zero structure of finite Blaschke products. As in the polynomial case, ``canonical" factorizations of 
a finite Blaschke product \eqref{9.6} exist in two cases: (1) when $B\in\R[[z]]$ is the product of real and spherical Blaschke factors,
and (2) when the nodes $(\alpha_1,\ldots,\alpha_n)$ in \eqref{9.6} form a spherical chain, in which case the factorization \eqref{9.6}
is unique. Otherwise, a factorization \eqref{9.6} for $B$ is largely non-unique.
It seems useful to have intrinsic characterizations of finite Blaschke products which do not appeal to their factorizations.
We start with the converse to Proposition \ref{P:8.1}: if $f\in\mathcal H_R$ (for some $R>1$) is such that $|f^{\bl}(\gamma)|=1$ whenever $|\gamma|=1$,
then $f$ is a finite Blaschke product. Actually, the following stronger result (the quaternionic analogue of a Fatou result \cite{fatou}) holds true.
\begin{theorem}
Let $f\in\mathcal H_1$ be such that 
\begin{equation}
\lim_{|\gamma|\to 1^-}|f^{\bl}(\gamma)|=1, \quad\mbox{or equivalently,}\quad \lim_{|\gamma|\to 1^-}|f^{\br}(\gamma)|=1.
\label{126}
\end{equation}
Then $f$ is a finite Blaschke product.
\label{T:8.2}
\end{theorem}
\begin{proof}
By Lemma \ref{L:im3}, the function $\varphi(r)=\max_{|z|=r}|f^{\bl}(\gamma)|$ is non-decreasing.
Therefore, we conclude from the first condition in \eqref{126} that $|f^{\bl}(\gamma)|\le 1$ for all $\gamma\in\mathbb B_1$ 
and hence $f\in\mathcal S_{\mathbb H}$. Alternatively, using the right-sided version in Lemma \ref{L:im3}, we
may derive from the second equality in \eqref{126} that $|f^{\br}(\gamma)|\le 1$ for all $\gamma\in\mathbb B_1$ with the same conclusion
that $f\in\mathcal S_{\mathbb H}$.

\smallskip

The equivalence of two conditions in \eqref{126} is not immediate, since in general, $|f^{\bl}(\gamma)|\neq |f^{\br}(\gamma)|$.
While proving this equivalence, we may (and will) assume that $f$ belongs to $\mathcal S_{\mathbb H}$.
We start with the equality 
\begin{equation}
|f^{\bl}(\gamma)|^2+|f^{\bl}(\overline\gamma)|^2=|f^{\br}(\gamma)|^2+|f^{\br}(\overline{\gamma})|^2
\label{127}
\end{equation}
holding for any $\gamma\in\mathbb B_1$. To see the latter, we fix $\gamma\in\mathbb B_1$ and divide 
$f$ by $\cX_{[\gamma]}$ with the remainder:
$$
f(z)=\cX_{[\gamma]}(z)g(z)+cz+d,\quad c,d\in\mathbb H, \; \; g\in\mathcal H_1.
$$ 
Then we evaluate the latter equality at $\gamma$ and $\overline\gamma$ on the left and the right arriving at
$$
f^{\bl}(\gamma)=\gamma c+d,\quad f^{\bl}(\overline\gamma)=\overline\gamma c+d,\quad f^{\br}(\gamma)=
c \gamma +d,\quad f^{\br}(\overline\gamma)=c\overline\gamma +d.
$$
Making use of the latter equalities, we verify \eqref{127} as follows:
$$
|f^{\bl}(\gamma)|^2+|f^{\bl}(\overline\gamma)|^2=2|\gamma|^2|c|^2+2|d|^2+4\Re(\gamma)\Re(c\overline{d})
=|f^{\br}(\gamma)|^2+|f^{\br}(\overline\gamma)|^2.
$$
Let us now assume that the first condition in \eqref{126} is in force. By a compactness argument, for any 
$\varepsilon>0$, there is $\delta>0$ such that 
\begin{equation}
|f^{\bl}(\gamma)|>1-\varepsilon,\quad\mbox{whenever}\quad  1-\delta<|\gamma|<1.
\label{128}
\end{equation}
Keeping in mind that $|f^{\br}(\beta)|\le 1$ for all $\beta\in\mathbb B_1$ and assuming in addition that $\varepsilon<\frac{2}{7}$
(for the last step in the computation below) we conclude from \eqref{127} and \eqref{128}
\begin{align*}
1\ge |f^{\br}(\gamma)|^2&\ge |f^{\br}(\gamma)|^2+|f^{\br}(\overline{\gamma})|^2-1\\
&=|f^{\bl}(\gamma)|^2+|f^{\bl}(\overline\gamma)|^2-1\ge 2(1-\varepsilon)^2-1>(1-3\varepsilon)^2,
\end{align*}
from which the second equality in \eqref{126} follows. The converse implication in \eqref{126} is verified similarly.

\smallskip

It follows from \eqref{128} that the power series $ff^\sharp\in\R[[z]]$ has no zeros on the annulus $1-\delta<|z|<1$ and hence, 
it has finitely many zeros in $\D=\mathbb B_1\bigcap \C$ (otherwise, $ff^\sharp\equiv 0$ which contradicts \eqref{126}). 
Since $ff^\sharp$ is real, its zero set (in $\mathbb B_1$) is the union of finitely many similarity classes $V_1,\ldots,V_m$.

\smallskip

By Theorem \ref{T:1.1h}, $f$ can be factored as $f=B^f_{V_1} g_1$, where $B^f_{V_1}$ is the finite Blaschke product having all
zeros in $V_1$ and $g_1\in\mathcal S_{\mathbb H}$ has no zeros in $V_1$ but has zeros in $V_2,\ldots,V_m$. We then recursively apply
Theorem \ref{T:1.1h} to get Schur-class power series $g_2,\ldots, g_m$ from factorizations
$$
g_{k}=B^{g_k}_{V_{k+1}}g_{k+1}\quad\mbox{for}\quad k=1,\ldots,m-1.
$$
Note that all zeros of $g_k$ are contained in $V_{k+1},\ldots,V_m$ and that $g_m\in\mathcal S_{\mathbb H}$ has no zeros in $\mathbb B_1$.
Combining all the above factorizations gives
\begin{equation}
f=Bg,\quad\mbox{where}\quad B=B^f_{V_1}B^{g_1}_{V_2}B^{g_2}_{V_3}\cdots B^{g_{m-1}}_{V_m} \; \; \mbox{and} \; \; g=g_m.
\label{jan18}
\end{equation}
Thus, the Schur-class power series $g$ in \eqref{jan18} has no zeros in $\mathbb B_1$. We next show that it satisfies the same 
conditions \eqref{126} as $f$, i.e., 
\begin{equation}
\lim_{|\gamma|\to 1^-}|g^{\bl}(\gamma)|=1=\lim_{|\gamma|\to 1^-}|g^{\br}(\gamma)|.
\label{129}
\end{equation}
Since the equalities in \eqref{129} are equivalent, it is enough to verify the second one. Let us assume that this second equality
fails to be true. Then there is $\varepsilon>0$ and a sequence $\{\gamma_k\}\subset \mathbb B_1$ such that $|\gamma_k|\to 1$ as $k\to\infty$,
and $|g^{\br}(\gamma_k)|<1-\varepsilon$ for all $k\ge 1$. Then we also have from \eqref{jan18}
$|f^{\br}(\gamma_k)|<1-\varepsilon$ for all $k\ge 1$. Indeed, if $g^{\br}(\gamma_k)=0$, then $|f^{\br}(\gamma_k)|=0$, by formula \eqref{16b}.
If $g^{\br}(\gamma_k)\neq 0$, then by Proposition \ref{P:8.1} and again by formula \eqref{16b}, we have
$$
|f^{\br}(\gamma_k)|=|B^{\br}\big(g^{\br}(\gamma_k)\gamma_k g^{\br}(\gamma_k)^{-1}\big)g^{\br}(\gamma_k)|<|g^{\br}(\gamma_k)|<1-\varepsilon.
$$
The latter inequality contradicts the main assumption \eqref{126}, which completes the proof of \eqref{129}. As a consequence of 
\eqref{129} we also have 
\begin{equation}
\lim_{|\gamma|\to 1^-}|(gg^\sharp)(\gamma)|=1.
\label{130}
\end{equation}
Indeed, it follows from \eqref{129} that for any $\varepsilon>0$, there is $\delta>0$ such that 
$|g^{\bl}(\gamma)|>1-\varepsilon$ {\em and} $|g^{\br}(\gamma)|>1-\varepsilon$ whenever $1-\delta<|\gamma|<1$.
Since $|g^{\bl}(\gamma)^{-1}\gamma g^{\bl}(\gamma)|=|\gamma|$, we have, by \eqref{16a} and \eqref{19},
\begin{align*}
|(gg^\sharp)(\gamma)|&=|g^{\bl}(\gamma)|\cdot |g^{\sharp\bl}\big(g^{\bl}(\gamma)^{-1}\gamma g^{\bl}(\gamma)\big)|\\
&=|g^{\bl}(\gamma)|\cdot |g^{\br}\big(\overline{g^{\bl}(\gamma)^{-1}\gamma g^{\bl}(\gamma)}\big)|>(1-\varepsilon)^2
>1-2\varepsilon,
\end{align*}
whenever $1-\delta<|\gamma|<1$, which justifies \eqref{130}.

\smallskip

Since $g$ has no zeros in $\mathbb B_1$, its formal inverse $g^{-1}:=\frac{1}{g}$ also belongs to $\mathcal H_1$,
and hence it can be evaluated on the left and on the right at every $\gamma\in\mathbb B_1$. From the first equality in \eqref{2.10} 
we see that 
$$
\frac{1}{g}=g^\sharp(gg^\sharp)^{-1}=(gg^\sharp)^{-1}g^\sharp,
$$ 
which together with \eqref{129} and \eqref{130} leads us to 
\begin{equation}
\lim_{|\gamma|\to 1^-}\left|\left(\frac{1}{g}\right)^{\bl}(\gamma)\right|=1=\lim_{|\gamma|\to 1^-}\left|\left(\frac{1}{g}\right)^{\br}(\gamma)\right|.
\label{131}
\end{equation}
Since $f\in\mathcal H_1$, we conclude from \eqref{131} by the maximum modulus principle (as we did for $f$ at the beginning of the proof)
that $\frac{1}{g}$ belongs to the Schur class $\mathcal S_\mathbb H$. Therefore, $\left|\left(\frac{1}{g}\right)^{\bl}(\gamma)\right|<1$
for all $\gamma\in\mathbb B_1$. Since the evaluation functionals are multiplicative (and equal) at real points (see Remark \ref{R:im7}),
we have $\frac{1}{g}(x)=g(x)^{-1}$ for any $x\in\mathbb B_1\bigcap \R=(-1,1)$. Therefore, for each $x\in(-1,1)$, we have 
$$
|g(x)|\le 1\quad\mbox{and}\quad \left|\left(\frac{1}{g}\right)(x)\right|=|g(x)|^{-1}\le 1
$$
which implies that $|g(x)|$ attains relative maximum inside $\mathbb B_1$ and therefore, it is a unimodular constant, by Lemma \ref{L:im3}.
We now get back to the formula \eqref{jan18}. Since $B$ is a finite Blaschke product and $g$ is  unimodular constant, it follows 
that $f=Bg$ is a finite Blaschke product.\qed
\end{proof}
In the classical complex setting, the proof of Theorem \ref{T:8.2} is substantially shorter: by the maximum modulus principle, condition 
\eqref{126} implies that $f$ is a Schur-class function with finitely many zeros; splitting off all zeros by a suitable finite Blaschke 
product $B$, we arrive at a Schur-class function $g=B^{-1}f$ having no zeros in $\D$ and satisfying condition \eqref{129}. Then its 
reciprocal is analytic in $\D$ and satisfies the same condition \eqref{131}. Then $1/g$ belongs to $\mathcal S$, by the maximum modulus 
principle. Then $|g(z)|=1$ for all $z\in\D$ and hence, $g$ is a unimodular constant, again by the maximum modulus principle. 

\smallskip

We conclude this section with another characterization of finite Blaschke products. To present it,
let us recall that an element $\alpha\in\bH$ is said to be a {\em right eigenvalue} of a matrix $A\in\mathbb H^{n\times n}$
if $A{\bf x}={\bf x}\alpha$ for some nonzero ${\bf x}\in{\mathbb H}^{n\times 1}$. In this case,
for any $\beta=h^{-1}\alpha h\sim \alpha$ we also have $A{\bf x}h=  {\bf x}hh^{-1}\alpha h={\bf x}h\beta$
and hence, any element in the similarity class $[\alpha]$ is a right eigenvalue of $A$. Therefore,
the right spectrum $\sigma_{\bf r}(A)$ of $A$ is the union of disjoint conjugacy classes (some of which
may be real singletons). 
It follows from the (quaternionic) canonical Jordan form  \cite{wieg} that 
${\displaystyle\lim_{k\to \infty} A^k}= 0$ if and only if $\sigma_{\bf r}(A)\subset\mathbb B_1$, in which case the matrix $A$ is called 
{\em stable}.	
\begin{theorem}
Given a unitary matrix $\sbm{A & B \\ C & D}\in\mathbb H^{(n+1)\times(n+1)}$ with $D\in\mathbb H$ and
$\sigma_{\bf r}(A)\subset\mathbb B_1$, the power series
\begin{equation}
f(z)=D+zC(I-zA)^{-1}B=D+\sum_{k=1}^\infty z^kCA^{k-1}B.
\label{132}
\end{equation}
is a Blaschke product of degree $n$. Conversely, any finite Blaschke product $f$ admits a stable unitary realization
of the form \eqref{132}.
\label{L:feb2}
\end{theorem}
\begin{proof}
For $f$ of the form \eqref{132} and a fixed $\gamma\in\mathbb B_1$,
\begin{equation}
f^{\bl}(\gamma)=D+\gamma \Upsilon(\gamma)B,\quad\mbox{where}\quad \Upsilon(\gamma):=\sum_{k=0}^\infty \gamma^{k}CA^{k},
\label{ups}
\end{equation}
and observe that the latter series converges for any $\gamma\in\overline{\mathbb B}_1$,
due to the assumption $\sigma_{\bf r}(A)\subset\mathbb B_1$. We next use the equalities
\begin{equation}
I-DD^*=CC^*,\quad BD^*=-AC^*,\quad BB^*=I-AA^*
\label{ups1}
\end{equation}
expressing the fact that $\sbm{A & B \\ C & D}$ is unitary, and then the equality
$\alpha \Upsilon(\alpha)A=\Upsilon(\alpha)-C$ (which is immediate from \eqref{ups}) to compute for
any $\alpha,\beta\in\mathbb 1$,
\begin{align}
1-f^{\bl}(\alpha)\overline{f^{\bl}(\beta)}&=1-(D+\alpha \Upsilon(\alpha)B)(D^*+B^*\Upsilon(\beta)^*\overline{\beta})\notag\\
&=CC^*+\alpha \Upsilon(\alpha)AC^*+CA^{*}\Upsilon(\beta)^*\overline{\beta})
-\alpha \Upsilon(\alpha)(I-AA^*)\Upsilon(\beta)^*\overline{\beta}\notag\\
&=CC^*+(\Upsilon(\alpha)-C)C^*+C(\Upsilon(\beta)^*-C^*)\notag\\
&\quad-\alpha \Upsilon(\alpha)\Upsilon(\beta)^*\overline{\beta}+(\Upsilon(\alpha)-C)(\Upsilon(\beta)^*-C^*)\notag\\
&=\Upsilon(\alpha)\Upsilon(\beta)^*-\alpha\Upsilon(\alpha)\Upsilon(\beta)^*\overline{\beta}.
\label{133}
\end{align}
Letting $\beta=\alpha$ in the latter equality and taking into account that $\Upsilon(\alpha)\Upsilon(\alpha)^*$ is a real number, we get
$$
1-|f^{\bl}(\alpha)|^2=(1-|\alpha|^2)\Upsilon(\alpha)\Upsilon(\alpha)^*,
$$
from which we see that $|f^{\bl}(\alpha)|\le 1$ if $|\alpha|<1$ and $|f^{\bl}(\alpha)|=1$ whenever $|\alpha|=1$. Therefore, $f$ is a 
finite Blaschke product, by Theorem \ref{T:8.2}. To complete the proof of the first statement, it remains to show that $\deg f=n$.
This will be done in the next section.

\smallskip

To prove the converse statement, let us take $f$ in the form \eqref{9.6} with $\alpha_1,\ldots,\alpha_n\in\mathbb B_1$ and 
let us recursively introduce the matrices
$U_k=\sbm{A_k & B_k \\ C_k & D_k}\in\mathbb H^{(k+1)\times (k+1)}$ for $k=1,\ldots,n$ as follows:
\begin{align}
&A_1=\overline{\alpha}_1,\quad B_1=C_1=\sqrt{1-|\alpha_1|^2},\quad D=-\alpha_1,\label{dop2}\\
&A_{k+1}=\begin{bmatrix}A_k & \sqrt{1-|\alpha_{k+1}|^2}B_k \\ 0 & \overline{\alpha}_{k+1}\end{bmatrix},\quad B_{k+1}=\begin{bmatrix}-B_k\alpha_{k+1}\\
\sqrt{1-|\alpha_{k+1}|^2}\end{bmatrix},\label{dop1}\\
&C_{k+1}=\begin{bmatrix}C_k & \sqrt{1-|\alpha_{k+1}|^2}D_k\end{bmatrix},\quad D_{k+1}=-D_k\alpha_{k+1}.\notag
\end{align}
By construction, the matrix $A_k$ is upper triangular with diagonal entries $\alpha_1,\ldots\alpha_k$; thus, $\sigma_{\bf r}(A_k)\subset\mathbb B_1$
for all $k=1,\ldots,n$. We next observe from \eqref{dop1} the equality
$$
\begin{bmatrix}A_{k+1} & B_{k+1} \\ C_{k+1} & D_{k+1}\end{bmatrix}
\begin{bmatrix}A^*_{k+1} & C_{k+1}^* \\ B^*_{k+1} & D_{k+1}^*\end{bmatrix}=
\begin{bmatrix}A_{k}A_{k}^* &  \quad 0 \quad&  A_kC_k^*+B_kD_k^* \\ 0 &  1&  0 \\ C_kA_k^*+D_kB_k^* & 0 & C_kC_k^*+D_kD_k^*  \end{bmatrix},
$$
from which it follows that the matrix $U_{k+1}$ is unitary if $U_k$ is unitary. Since $U_1$ is clearly unitary, it follows by induction that
$U_k$ is unitary for all $k=1,\ldots,n$. Letting
$$
f_k(z):=D_k+zC_k(I-zA_k)^{-1}B_k \quad\mbox{for}\quad k=1,\ldots,n,
$$
and observing from \eqref{dop2} that 
\begin{equation}
f_1(z):=-\alpha_1+(1-|\alpha_1|^2)(1-z\overline{\alpha}_1)^{-1}={\bf b}_{\alpha_{1}}(z),
\label{dop3}
\end{equation}
we then compute, upon making use of \eqref{dop1},
\begin{align*}
f_{k+1}(z)&=D_{k+1}+zC_{k+1}(I-zA_{k+1})^{-1}B_{k+1}\notag\\ 
&=-D_k\alpha_{k+1}+z\begin{bmatrix}C_k & \sqrt{1-|\alpha_{k+1}|^2}D_k\end{bmatrix}\notag\\
&\quad\times\begin{bmatrix}(1-zA_k)^{-1} & z\sqrt{1-|\alpha_{k+1}|^2}(1-zA_k)^{-1}B_k 
(1-z\overline{\alpha}_{k+1})^{-1}\\ 0 & (1-z\overline{\alpha}_{k+1})^{-1}\end{bmatrix}\notag\\
&\quad\times \begin{bmatrix}-B_k\alpha_{k+1}\\
\sqrt{1-|\alpha_{k+1}|^2}\end{bmatrix} \notag\\
&=-D_k\alpha_{k+1}-zC_k(1-zA_k)^{-1}B_k\alpha_{k+1}+z(1-|\alpha_{k+1}|^2)D_k(1-z\overline{\alpha}_{k+1})^{-1}\notag\\
&\quad +z(1-|\alpha_{k+1}|^2)C_k(1-zA_k)^{-1}B_k(1-z\overline{\alpha}_{k+1})^{-1}\notag\\
&=\big(D_k+zC_k(I-zA_k)^{-1}B_k \big)(z-\alpha_{k+1})(1-z\overline{\alpha}_{k+1})^{-1}=f_k(z){\bf b}_{\alpha_{k+1}}(z).
\end{align*}
Combining the latter equality (for $k=1,\ldots,n-1$) with \eqref{dop3}, we conclude that 
\begin{equation}
D_n\phi+zC_n(I-zA_n)^{-1}B_n\phi={\bf b}_{\alpha_1}(z){\bf b}_{\alpha_2}(z)\cdots {\bf b}_{\alpha_n}(z)\phi=f(z),
\label{dop7}
\end{equation}
and since the matrix $\sbm{A_n & B_n\phi \\ C_n & D_n\phi}=U_n\sbm{I_n & 0 \\ 0 & \phi}$ is unitary and 
$\sigma_{\bf r}(A_n)\subset\mathbb B_1$, the proof of the converse part is complete.\qed
\end{proof}
\begin{Rk}
{\rm In the complex setting, the matrices \eqref{dop1} appeared in \cite{young}, and in the free-coordinate form, go back to 
\cite{dBR1,dBR2}. If we considers the Hardy space \eqref{dop4} as the right Hilbert $\mathbb H$-module with inner product
\begin{equation}
\langle h, \, g\rangle_{H^2(\mathbb B_1)}=\sum_{k=0}^\infty \overline{g}_kh_k,\quad\mbox{where}\quad h(z)=\sum_{k=0}^\infty h_kz^k, \; \; 
g(z)=\sum_{k=0}^\infty g_kz^k,
\label{dop6}
\end{equation}
it turns out that for any finite Blaschke product $f={\bf b}_{\alpha_1}{\bf b}_{\alpha_2}\cdots {\bf b}_{\alpha_{n}}$,
the set 
\begin{equation}
\sqrt{1-|\alpha_1|^2}{\bf k}_{\alpha_1}, \; \; \sqrt{1-|\alpha_j|^2}
{\bf b}_{\alpha_1}{\bf b}_{\alpha_2}\cdots {\bf b}_{\alpha_{j-1}}{\bf k}_{\alpha_{j}} \; \; (j=2,\ldots k)
\label{dop5}
\end{equation}
is orthonormal (in the above metric) and its right linear span coincides with the orthogonal complement to the 
right Hilbert $\mathbb H$-module $f\cdot H^2(\mathbb B_1)$. Since the latter is invariant under the operator 
$M_z$ of multiplication by $z$, its orthogonal complement is invariant under the adjoint operator (in metric \eqref{dop6})
\begin{equation}
M_z^*:=R_0: \; f(z)=\sum_{k=0}^\infty f_kz^k \mapsto f(z)=\sum_{k=0}^\infty f_{k+1}z^k.
\label{dop7a}
\end{equation}
Computing the matrices of the operator $R_0$ and of the functional $E: \, h\to h_0$ on
$(f\cdot H^2(\mathbb B_1))^\perp$ with respect to the orthonormal basis \eqref{dop5} (the quaternionic {\em Takenaka basis} \cite{take})
we come up with $A_n$ and $C_n$ constructed in \eqref{dop1}. Furthermore, the operator $R_0 M_f$ maps $\mathbb H$ into 
$(f\cdot H^2(\mathbb B_1))^\perp$, and its matrix representation with respect to the basis \eqref{dop5} equals $B_n$.
Hence, the realization \eqref{dop7} is the explicit form of the canonical de Branges-Rovnyak realization (for a finite Blaschke product).
We refer to \cite{acsbook} for further results in this fashion.}
\label{R:4.8}
\end{Rk}

\section{Finite Blaschke products with prescribed zero structure}
\setcounter{equation}{0}
Given a finite Blaschke product $f$ with all zeros contained within spheres $V_1,\ldots, V_m$, we can find its 
spherical divisors $D^f_{\ell,V_k}$, $D^f_{r,V_k}$ as in Theorem \ref{T:1.1} or hyperbolic spherical divisors $\mathcal B^f_{\ell,V_k}$, 
$\mathcal B^f_{r,V_k}$ as in Theorem \ref{T:1.1h}. In this section we will consider the inverse problem -- to construct a finite Blaschke product 
with prescribed spherical divisors. It is easy to take care of real and spherical zeros by taking 
real Blaschke factors and spherical Blaschke factors with suitable powers. All these factors are power series with real coefficients
and commute with any other power series in $\mathbb H[[z]]$. Thus, the main part of the above question is concerned about Blaschke products
with no real and spherical zeros. The case where all spherical divisors are of degree one, i.e., $\mathcal B^f_{\ell,V_k}={\bf b}_{\alpha_k}$
was handled in \cite{acs} upon making use of formula \eqref{16a} as follows: a finite Blaschke product having left zeros at given pairwise 
non-similar elements $\alpha_1,\ldots,\alpha_m\in\mathbb B_1$ can be taken in the form
$$
f={\bf b}_{\beta_1}{\bf b}_{\beta_2}\cdots {\bf b}_{\beta_m},
$$
where the nodes $\beta_1,\ldots,\beta_m$ are defined recursively by $\beta_1=\alpha_1$ 
\begin{equation}
\beta_{k}=(({\bf b}_{\beta_1}\cdots {\bf b}_{\beta_{k-1}})^{\bl}(\alpha_k))^{-1} \alpha_k  
({\bf b}_{\beta_1}\cdots {\bf b}_{\beta_{k-1}})^{\bl}(\alpha_k)
\label{jan20}
\end{equation}
for $k=2,\ldots,m$. This approach still applies to a more general case where one of the spherical divisors is 
based on a spherical chain.  To point out the difficulties arising in the case when two prescribed spherical divisors are of degree greater than one, 
we consider the following example.
\begin{Ex}
{\rm Given two non-similar and non-real $\alpha,\beta\in\mathbb B_1$, find a Blaschke product $B$ ($\deg B=4$)
with spherical divisors $\bp^2_\alpha$ and $\bp^2_\beta$ (or ${\bf b}_{\alpha}^2$ and ${\bf b}_{\alpha}^2$)}.
\label{E:jan12}
\end{Ex}
The first step in the recursion \eqref{jan20} suggests the Blaschke product
$$
{\bf b}^2_\alpha {\bf b}_{\beta_1},\quad\mbox{where}\quad \beta_1=(({\bf b}^2_\alpha)^{\bl}(\beta))^{-1}\beta (({\bf b}^2_\alpha)^{\bl}(\beta)) 
$$ 
which indeed has a left zero at $\beta$. The next step as described in \eqref{jan20} cannot be implemented since 
$({\bf b}^2_\alpha {\bf b}_{\beta_1})^{\bl}(\beta)=0$. Of course, the desired $B$ exists and turns out to be of the form 
$P={\bf b}^2_\alpha {\bf b}_{\beta_1}{\bf b}_{\beta_2}$ for an appropriately chosen $\beta_2\in[\beta]$. Finding $\beta_2$ is not that easy, however.

\smallskip

The polynomial case is easier; various algorithm for constructing quaternionic polynomials with prescribed spherical divisors can be found in 
\cite{bolapp}. As an illustrative example, it is shown in \cite[Example 5.1]{bolapp} that 
the least right common multiple of the polynomials $\bp_\alpha^2$ and $\bp^2_\beta$ is given by the formula
\begin{equation}
{\bf lrcm}(\bp_\alpha^2,\bp_\beta^2)=(z-\alpha)^2(z-\beta_1)(z-\beta_2),
\label{1.14g}
\end{equation}
where 
\begin{align*}
\beta_1&=(\beta^2-2\beta\alpha+\alpha^2)^{-1}\beta (\beta^2-2\beta\alpha+\alpha^2),\\
\beta_2&=(3\beta^2-4\beta\alpha+\alpha^2+2(\alpha-\beta)\beta_1)^{-1}
\beta(3\beta^2-4\beta\alpha+\alpha^2+2(\alpha-\beta)\beta_1).
\end{align*}
Another result from \cite{bolapp} (see Theorem 6.1 there) relevant to the present setting is the following.
\begin{theorem}
Given any $\alpha_1,\ldots,\alpha_m\in\mathbb B_1$, there exist $\beta_1,\ldots,\beta_m$ ($\gamma_i\in[\alpha_i]$) so that
the power series
$$
f=\bp_{\alpha_1}\bp_{\alpha_2}\cdots \bp_{\alpha_m}{\bf k}_{\gamma_1}\cdot {\bf k}_{\gamma_2}\ldots{\bf k}_{\gamma_m}
$$
is a finite Blaschke product (of degree $m$).
\label{T:jan1}
\end{theorem} 
Since any monic polynomial can be factored into a product \eqref{1.10} of linear factors, Theorem \ref{T:jan1} actually states that 
given any monic polynomial $p$ with all zeros inside $\mathbb B_1$, there exists a power series $G\in\mathcal H_1$ having no zeros
and such that $B=pG$ is a finite Blaschke product. The proof of the theorem is constructive and provides an algorithm
for constructing the nodes $\gamma_1,\ldots,\gamma_m$. Applying this construction to the case $\alpha_1=\alpha_2=\alpha$, $\alpha_3=\beta_1$,
$\alpha_4=\beta_2$ in \eqref{1.14g}, one can get the explicit answer for Example \ref{E:jan12}. In the general case we do the same:
to construct a finite Blaschke product with prescribed (say, left) spherical divisors, we can find the {\bf lrcm} of these divisors, 
then factor the resulting polynomial as in \eqref{1.10} (this is doable since the conjugacy classes containing all zeros of this {\bf lrcm} are known) 
and then apply the algorithm from the proof of \cite[Theorem 6.1]{bolapp}.
Still, it is desirable to get the explicit formula for $G$ and $B$ in terms of the coefficients of $p$. 
Such formulas are presented in Theorem \ref{T:1} below. We start with some needed preliminaries.

\smallskip 
\noindent
For a monic polynomial $p\in\mathbb {\mathbb H}[z]$, the associated (left) 
{\em companion matrix} is defined by
\begin{equation}
C_p=\begin{bmatrix} 0 & 0 &\ldots & 0 &-p_0\\
1 & 0 & \ldots& 0 &-p_1\\
0 & 1& \ldots &0 & -p_2 \\
\vdots & \vdots &\ddots & \vdots & \vdots \\
0 & 0 &\ldots &1 & -p_{n-1} \end{bmatrix},\quad\mbox{if} \; \;  p(z)=z^n+\sum_{k=0}^{n-1} p_{k}z^{k}.
\label{1.14e}
\end{equation}
Let ${\bf e}_k$ denotes the $k$-th column of the $n\times n$ identity matrix $I_n$ and let
\begin{equation}
F=\begin{bmatrix} {\bf e}_2 & {\bf e}_3 & \ldots & {\bf e}_n\end{bmatrix}=
\begin{bmatrix} 0&0& \ldots&0\\ 1 &0&\ldots & 0 \\ \vdots&\ddots&\ddots&\vdots\\
0 & \ldots & 1 & 0\end{bmatrix}.
\label{1.14eg}
\end{equation}
Then we clearly have
\begin{equation}
p(z)=z^n-\begin{bmatrix}1 & z & \ldots & z^{n-1}\end{bmatrix}C^n_p{\bf e}_1=z^n-
{\bf e}_1^*(I-zF^*)^{-1}C^n_p{\bf e}_1.
\label{1.15e}
\end{equation}
The right spectrum of $C_p$ coincides with the zero set of the polynomial $pp^\sharp$ and therefore, $p$ has all zeros 
inside $\mathbb B_1$ if and only if  $\sigma_{\bf r}(C_p)\subset \mathbb B_1$.

\smallskip

If $A$ is any matrix similar to $C_p$, i.e., $A=TC_p T^{-1}$ for some invertible $T\in \mathbb H^{n\times n}$,
and if ${\bf v}=T{\bf e}_1$, then the {\em controllability matrix} $\mathfrak C_{A,{\bf v}}$ of the pair $(A,{\bf v})$, 
\begin{align}
\mathfrak C_{A,{\bf v}}&:=\begin{bmatrix}{\bf v} & A{\bf v} &\ldots & A^{n-1}{\bf v}\end{bmatrix}\label{feb1}\\
&=T\begin{bmatrix}{\bf e}_1& C_p{\bf e}_1 &\ldots & C_p^{n-1}{\bf e}_1\end{bmatrix}
=T\begin{bmatrix}{\bf e}_1& {\bf e}_2&\ldots & {\bf e}_n\end{bmatrix}=T\notag
\end{align}
is invertible meaning that the pair  $(A,{\bf v})$ is {\em controllable}. We then have equalities 
$$
C^n_p{\bf e}_1=T^{-1}A^nTT^{-1}{\bf v}=\mathfrak C_{A,{\bf v}}^{-1}A^n{\bf v},
$$
which allows us to conclude that the polynomial \eqref{1.15e} can be represented as 
\begin{equation}
p(z)=z^n-{\bf e}_1^*(I-zF^*)^{-1}\mathfrak C_{A,{\bf v}}^{-1}A^n{\bf v}
\label{1.15g}
\end{equation}
for any pair $(A,{\bf v})$ similar to the pair $(C_p,{\bf e_1})$ in the sense that $A=TC_p T^{-1}$ and ${\bf v}=T{\bf e}_1$ for 
some invertible $T\in \mathbb H^{n\times n}$. 

\smallskip

Since the right spectra of similar matrices coincide,  $\sigma_{\bf r}(A)\subset\mathbb B_1$ if all zeros 
of $p(z)$ are contained $\mathbb B_1$. In this case, the Stein equation 
\begin{equation}
P_{A,{\bf v}}-AP_{A,{\bf v}}A^*={\bf v}{\bf v}^*
\label{12u}
\end{equation}
has a unique solution which is given by converging series
\begin{equation}
P_{A,{\bf v}}=\sum_{k=0}^\infty A^k{\bf v}{\bf v}^*A^{*k}.
\label{7u}
\end{equation}
Since ${\bf v}{\bf v}^*+A{\bf v}{\bf v}^*A^*+\ldots+A^{n-1}{\bf v}{\bf v}^*A^{*(n-1)}=\mathfrak C_{A,{\bf v}}$, by \eqref{feb1}, we have
\begin{equation}
P_{A,{\bf v}}=\sum_{k=0}^{n-1} A^k{\bf v}{\bf v}^*A^{*k}+A^n\sum_{k=0}^\infty A^k{\bf v}{\bf v}^*A^{*k}A^{*n}
=\mathfrak C_{A,{\bf v}}\mathfrak C_{A,{\bf v}}^*+A^n P_{A,{\bf v}}A^{*n},
\label{7s}
\end{equation}
from which it follows that $P_{A,{\bf v}}$ is positive definite. Since $\sigma_{\bf r}(A)\subset\mathbb B_1$ 
the matrix $(I-A)$ is invertible. Let us define the vector 
\begin{equation}
{\bf g}_{A,{\bf v}}=(I-A^*)P_{A,{\bf v}}^{-1}(I-A)^{-1}{\bf v}.
\label{6u}
\end{equation}
Multiplying both parts of \eqref{12u} by the matrix $(I-A^*)P_{A,{\bf v}}^{-1}(I-A)^{-1}$ on the left and by its adjoint on the right,
we get (after simple matrix manipulations) $P_{A,{\bf v}}^{-1}-A^*P_{A,{\bf v}}^{-1}A$ on the left and ${\bf g}_{A,{\bf v}}{\bf g}_{A,{\bf v}}^*$
(by \eqref{6u}) on the right. Since $\sigma_{\bf r}(A)\subset\mathbb B_1$, we then conclude from the resulting equality 
\begin{equation}
P_{A,{\bf v}}^{-1}-A^*P_{A,{\bf v}}^{-1}A={\bf g}_{A,{\bf v}}{\bf g}_{A,{\bf v}}^*
\label{feb5}
\end{equation}
that $P_{A,{\bf v}}^{-1}$ can be represented by converging series  
\begin{equation}
P^{-1}_{A,{\bf v}}=\sum_{k=0}^\infty A^{*k}{\bf g}_{A,{\bf v}}{\bf g}_{A,{\bf v}}^*A^{k}.
\label{7v}
\end{equation}
We next observe the equality
\begin{equation}
AP_{A,{\bf v}}{\bf g}_{A,{\bf v}}=-{\bf v}+{\bf v}{\bf v}^*P_{A,{\bf v}}^{-1}(I-A)^{-1}{\bf v}.
\label{8u}
\end{equation}
Indeed, by \eqref{6u}, we have
\begin{align*}AP_{A,{\bf v}}{\bf g}_{A,{\bf v}}&=AP_{A,{\bf v}}(I-A^*)P_{A,{\bf v}}^{-1}(I-A)^{-1}{\bf v}\\
&=((A-I)P_{A,{\bf v}}+{\bf v}{\bf v}^*)P_{A,{\bf v}}^{-1}(I-A)^{-1}{\bf v},
\end{align*}
which verifies \eqref{8u}. 
\begin{theorem}
Let assume that a monic polynomial $p\in\mathbb H[z]$ of the form \eqref{1.15g} has all zeros inside $\mathbb B_1$
(i.e., $\sigma_{\bf r}(A)\subset\mathbb B_1$). Then
\begin{enumerate}
\item The power series 
\begin{equation}
R(z)={\bf e}_n^*\mathfrak C_{A,{\bf v}}^{-1}P_{A,{\bf v}}(I-zA^*)^{-1}{\bf g}_{A,{\bf v}}
=\sum_{k=0}^\infty z^k{\bf e}_n^*
\mathfrak C^{-1}_{A,{\bf v}}P_{A,{\bf v}}A^{*k}{\bf g}_{A,{\bf v}},
\label{9u}
\end{equation}
where $P_{A,{\bf v}}$ and ${\bf g}_{A,{\bf v}}$ are defined as in \eqref{7u} and \eqref{6u}, has no zeros in $\mathbb B_1$
and belongs to $H^2(\mathbb B_1)$ with 
$$
\|R\|_{H^2(\mathbb B_1)}^2={\bf e}_n^*\mathfrak C^{-1}_{A,{\bf v}}P_{A,{\bf v}}\mathfrak C^{-*}_{A,{\bf v}}{\bf e_n}.
$$ 
\item The power series $\Theta(z)=p(z)R(z)$ is a finite Blaschke product. Furthermore, $\Theta$ is defined explicitly by 
\begin{align}
\Theta(z)&=1-{\bf v}^*(I-A^*)^{-1}{\bf g}_{A,{\bf v}}+z{\bf v}^*(I-zA^*)^{-1}{\bf g}_{A,{\bf v}}
\label{14u}\\
&=1-{\bf v}^*(I-A^*)^{-1}{\bf g}_{A,{\bf v}}+\sum_{k=0}^\infty z^{k+1}{\bf v}^*A^{*k}{\bf g}_{A,{\bf v}}.\notag  
\end{align}
\end{enumerate}
\label{T:1}
\end{theorem}
\begin{proof}
Since $\sigma_{\bf r}(A)\subset\mathbb B_1$, the power series $R$ clearly belongs to $\mathcal H_1$. By the definition
of the norm in $H^2(\mathbb B_1)$, we have from \eqref{9u} and \eqref{7v},
\begin{align*}
\|R\|_{H^2(\mathbb B_1)}^2&=\sum_{k=0}^\infty |{\bf e}_n^*
\mathfrak C^{-1}_{A,{\bf v}}P_{A,{\bf v}}A^{*k}{\bf g}_{A,{\bf v}}|^2\\
&={\bf e}_n^* \mathfrak C^{-1}_{A,{\bf v}}P_{A,{\bf v}}
\bigg(\sum_{k=0}^\infty A^{*k}{\bf g}_{A,{\bf v}}{\bf g}_{A,{\bf v}}^*A^{k}\bigg)P_{A,{\bf v}}\mathfrak C^{-*}_{A,{\bf v}}{\bf e}_n\\
&={\bf e}_n^*\mathfrak C^{-1}_{A,{\bf v}}P_{A,{\bf v}}\mathfrak C^{-*}_{A,{\bf v}}{\bf e_n}.
\end{align*}
We break the rest of the proof into three (quite independent) steps.

\medskip
\noindent
{\bf Step 1:} {\em The product $\Theta=pR$ is indeed of the form \eqref{14u}}.  

\smallskip

To this end, we first verify that 
\begin{equation}
{\bf v}p(z)=(zI-A)\mathfrak C_{A,{\bf v}}(I-zF^*)^{-1}\big({\bf e}_n-F^*\mathfrak C_{A,{\bf v}}^{-1}A^n{\bf v}\big).
\label{3u}
\end{equation}
To this end we observe the equalities 
\begin{equation}
A\mathfrak C_{A,{\bf v}}-\mathfrak C_{A,{\bf v}}F=A^n{\bf v}{\bf e}_n^*,\quad
\mathfrak C_{A,{\bf v}}-A\mathfrak C_{A,{\bf v}}F^*={\bf v}{\bf e}_1^*.
\label{4u}
\end{equation}
which follow directly from the definition \eqref{feb1} of $\mathfrak C_{A,{\bf v}}$. 
Multiplying the second equality in \eqref{4u} by $z(I-zF^*)^{-1}$ on the right and making use of the identity 
\begin{equation}
zF^*(I-zF^*)^{-1}=(I-zF^*)^{-1}-I
\label{5u}
\end{equation}
we get 
\begin{equation}
(zI-A)\mathfrak C_{A,{\bf v}}(I-zF^*)^{-1}=-A\mathfrak C_{A,{\bf v}}+z{\bf v}{\bf e}_1^*(I-zF^*)^{-1}.
\label{p6u}
\end{equation}
Taking into account \eqref{4u}, \eqref{5u}, \eqref{p6u} as well as the equalities 
$$
\mathfrak C_{A,{\bf v}}{\bf e}_n=A^{n-1}{\bf v},\quad {\bf e}_1^*(I-zF^*)^{-1}{\bf e}_n=z^{n-1},
$$
we simplify the right side expression in \eqref{3u} as follows:
\begin{align*}
&(zI-A)\mathfrak C_{A,{\bf v}}(I-zF^*)^{-1}\left({\bf e}_n-F^*\mathfrak C_{A,{\bf v}}^{-1}A^n{\bf v}\right)\\
&=\left(-A\mathfrak C_{A,{\bf v}}+z{\bf v}{\bf e}_1^*(I-zF^*)^{-1}\right)
\left({\bf e}_n-F^*\mathfrak C_{A,{\bf v}}^{-1}A^n{\bf v}\right)\\
&=-A^n{\bf v}+z^n{\bf v}+(\mathfrak C_{A,{\bf v}}-{\bf v}{\bf e}_1^*)\mathfrak C_{A,{\bf v}}^{-1}A^n{\bf v}
-{\bf v}{\bf e}_1^*((I-zF^*)^{-1}-I)\mathfrak C_{A,{\bf v}}^{-1}A^n{\bf v}\notag\\
&=z^n{\bf v}-{\bf v}{\bf e}_1^*(I-zF^*)^{-1}\mathfrak C_{A,{\bf v}}^{-1}A^n{\bf v}={\bf v}p(z),
\end{align*}
which verifies \eqref{3u}. We next recall the first equality in \eqref{4u} and observe that 
${\bf e}_n{\bf e}_n^*=I-F^*F$ to conclude that
\begin{align}
\big({\bf e}_n-F^*\mathfrak C_{A,{\bf v}}^{-1}A^n{\bf v}\big){\bf e}_n^*\mathfrak C_{A,{\bf v}}^{-1}&=
(I-F^*F)\mathfrak C_{A,{\bf v}}^{-1}-F^*\mathfrak C_{A,{\bf v}}^{-1}(A\mathfrak C_{A,{\bf v}}-\mathfrak C_{A,{\bf v}}F)
\mathfrak C_{A,{\bf v}}^{-1}\notag\\
&=\mathfrak C_{A,{\bf v}}^{-1}-F^*\mathfrak C_{A,{\bf v}}^{-1}A.\label{feb11}
\end{align}
Combining the latter equality with \eqref{3u} and making use of \eqref{p6u} and of the second equality in \eqref{4u} we get
\begin{align*}
{\bf v}p(z){\bf e}_n^*\mathfrak C_{A,{\bf v}}^{-1}&=(zI-A)\mathfrak C_{A,{\bf v}}(I-zF^*)^{-1}
\big(\mathfrak C_{A,{\bf v}}^{-1}-F^*\mathfrak C_{A,{\bf v}}^{-1}A\big)\\
&=-A\mathfrak C_{A,{\bf v}}\big(\mathfrak C_{A,{\bf v}}^{-1}-F^*\mathfrak C_{A,{\bf v}}^{-1}A\big)\\
&\quad+z{\bf v}{\bf e}_1^*(I-zF^*)^{-1}\big(\mathfrak C_{A,{\bf v}}^{-1}-F^*\mathfrak C_{A,{\bf v}}^{-1}A\big)\\
&=-{\bf v}{\bf e}_1^*\mathfrak C_{A,{\bf v}}^{-1}A
+z{\bf v}{\bf e}_1^*(I-zF^*)^{-1}\big(\mathfrak C_{A,{\bf v}}^{-1}-F^*\mathfrak C_{A,{\bf v}}^{-1}A\big)\\
&={\bf v}{\bf e}_1^*(I-zF^*)^{-1}\big(-(I-zF^*)\mathfrak C_{A,{\bf v}}^{-1}A+z
\big(\mathfrak C_{A,{\bf v}}^{-1}-F^*\mathfrak C_{A,{\bf v}}^{-1}A\big)\big)\\
&={\bf v}{\bf e}_1^*(I-zF^*)^{-1}\mathfrak C_{A,{\bf v}}^{-1}(zI-A).
\end{align*}
Combining the latter equality with \eqref{9u} we get
\begin{equation}
{\bf v}p(z)R(z)={\bf v}{\bf e}_1^*(I-zF^*)^{-1}\mathfrak C_{A,{\bf v}}^{-1}(zI-A)P_{A,{\bf v}}
(I-zA^*)^{-1}{\bf g}_{A,{\bf v}}.
\label{10u}
\end{equation}
By \eqref{12u}, \eqref{p6u} and \eqref{8u},
\begin{align}
&(zI-A)P_{A,{\bf v}} (I-zA^*)^{-1}{\bf g}_{A,{\bf v}}\notag\\
&=-AP_{A,{\bf v}}{\bf g}_{A,{\bf v}}+z{\bf v}{\bf v}^*(I-zA^*)^{-1}{\bf g}_{A,{\bf v}}\notag\\
&={\bf v}-{\bf v}{\bf v}^*P_{A,{\bf v}}^{-1}(I-A)^{-1}{\bf v}+z{\bf v}{\bf v}^*(I-zA^*)^{-1}{\bf g}_{A,{\bf v}}\notag\\
&={\bf v}\big(1-{\bf v}^*(I-A^*)^{-1}{\bf g}_{A,{\bf v}}+z{\bf v}^*(I-zA^*)^{-1}{\bf g}_{A,{\bf v}}\big)
={\bf v}\Theta(z),\label{feb13}
\end{align}
where the last step is clear from \eqref{14u}.
We now substitute the latter equality into \eqref{10u} and take into account that 
$\mathfrak C_{A,{\bf v}}^{-1}{\bf v}={\bf e}_1$ and ${\bf e}_1^*(I-zF^*)^{-1}{\bf e_1}=1$:
\begin{align*}
{\bf v}p(z)R(z)&={\bf v}{\bf e}_1^*(I-zF^*)^{-1}\mathfrak C_{A,{\bf v}}^{-1}
{\bf v}\Theta(z)\\
&={\bf v}{\bf e}_1^*(I-zF^*)^{-1}{\bf e_1}\Theta(z)={\bf v}\Theta(z).
\end{align*}
Since ${\bf v}=T{\bf e}_1\neq 0$, the latter equality implies $p(z)R(z)=\Theta(z)$. 

\medskip
\noindent
{\bf Step 2:} {\em The power series $\Theta$ of the form \eqref{14u} is a finite Blaschke product}. 

\smallskip
\noindent
Let $D=1-{\bf v}^*(I-A^*)^{-1}{\bf g}_{A,{\bf v}}$ denote the free coefficient in \eqref{14u} and let
$$
\widetilde A=P_{A,{\bf v}}^{\frac{1}{2}}A^*P_{A,{\bf v}}^{-\frac{1}{2}},\quad B=P_{A,{\bf v}}^{\frac{1}{2}}{\bf g}_{A,{\bf v}},
\quad C={\bf v}^*P_{A,{\bf v}}^{-\frac{1}{2}}.
$$
Then we have 
$C\widetilde{A}^kB={\bf v}^*A^{*k}{\bf g}_{A,{\bf v}}$ for all $k\ge 0$; hence, $\Theta$ of the form \eqref{14u}
also can be written as 
\begin{equation}
\Theta(z)=D+zC(I-z\widetilde{A})^{-1}B.
\label{feb8}
\end{equation}
To complete Step 2, it suffices (by Theorem \ref{L:feb2}) to show that the matrix $\sbm{\widetilde{A} & B \\ C & D}$ is unitary. 
In other words we need to verify the equalities 
$\widetilde{A}\widetilde{A}^*+BB^*=I$, $\widetilde{A}C^*+BD^*=0$ and $CC^*+DD^*=1$. This is done below,
upon making multiple use of formulas \eqref{feb5} and \eqref{6u}: 
\begin{align*}
\widetilde{A}\widetilde{A}^*+BB^*&=P_{A,{\bf v}}^{\frac{1}{2}}\big(A^*P_{A,{\bf v}}^{-1}A
+{\bf g}_{A,{\bf v}}{\bf g}_{A,{\bf v}}\big)P_{A,{\bf v}}^{\frac{1}{2}}=
P_{A,{\bf v}}^{\frac{1}{2}}P_{A,{\bf v}}^{-1}P_{A,{\bf v}}^{\frac{1}{2}}=I,\\
\widetilde{A}C^*+BD^*&=P_{A,{\bf v}}^{\frac{1}{2}}\big(A^*P^{-1}_{A,{\bf v}}{\bf v}+
{\bf g}_{A,{\bf v}}\big(1-{\bf g}_{A,{\bf v}}^*(I-A)^{-1}{\bf v}\big)\big)\\
&=P_{A,{\bf v}}^{\frac{1}{2}}\big(A^*P^{-1}_{A,{\bf v}}{\bf v}+
{\bf g}_{A,{\bf v}}-(P^{-1}_{A,{\bf v}}-A^*P^{-1}_{A,{\bf v}}A)(I-A)^{-1}{\bf v})\\\
&=P_{A,{\bf v}}^{\frac{1}{2}}\big({\bf g}_{A,{\bf v}}-(I-A^*)P_{A,{\bf v}}^{-1}(I-A)^{-1}{\bf v}\big)=0,\\
1-DD^*&=1-(1-{\bf v}^*(I-A^*)^{-1}{\bf g}_{A,{\bf v}})(1-{\bf g}_{A,{\bf v}}^*(I-A)^{-1}{\bf v})\\
&={\bf v}^*(I-A^*)^{-1}{\bf g}_{A,{\bf v}}+{\bf g}_{A,{\bf v}}^*(I-A)^{-1}{\bf v}\\
&\quad-{\bf v}^*(I-A^*)^{-1}(P^{-1}_{A,{\bf v}}-A^*P^{-1}_{A,{\bf v}}A)(I-A)^{-1}{\bf v})\\
&={\bf v}^*P_{A,{\bf v}}^{-1}(I-A)^{-1}{\bf v}
+{\bf v}^*(I-A^*)^{-1}P_{A,{\bf v}}^{-1}{\bf v}\\
&\quad-{\bf v}^*(I-A^*)^{-1}P^{-1}_{A,{\bf v}}(I-A)^{-1}{\bf v}\\
&\quad +{\bf v}^*\big((I-A^*)^{-1}-I\big)P^{-1}_{A,{\bf v}}\big((I-A)^{-1}-I\big){\bf v}\\
&={\bf v}^*P_{A,{\bf v}}^{-1}{\bf v}=CC^*.
\end{align*}
{\bf Step 3:} {\em The power series $R$ given by the formula \eqref{9u} has no zeros in $\mathbb B_1$}. 

\smallskip

We will show that the polynomial (of degree $n$) 
\begin{equation}
G(z)=p(z)-(z-1){\bf g}_{A,{\bf v}}^*(I-A)^{-1}\mathfrak C_{A,{\bf v}}(I-zF^*)^{-1}\big({\bf e}_n-F^*\mathfrak C_{A,{\bf v}}^{-1}A^n{\bf v}\big)
\label{feb10}
\end{equation}
is the formal inverse of $R$, The existence of such an inverse implies the absence of zeros for $R$. 
We start our last computation: since 
$pR=\Theta$, by Step 1, we have from \eqref{9u} and \eqref{feb10}
\begin{align}
G(z)R(z)&=\Theta(z)-(z-1){\bf g}_{A,{\bf v}}^*(I-A)^{-1}\mathfrak C_{A,{\bf v}}(I-zF^*)^{-1}\notag\\
&\qquad\qquad\quad \times
\big({\bf e}_n-F^*\mathfrak C_{A,{\bf v}}^{-1}A^n{\bf v}\big){\bf e}_n^*
\mathfrak C_{A,{\bf v}}^{-1}P_{A,{\bf v}}(I-zA^*)^{-1}{\bf g}_{A,{\bf v}}\notag\\
&=\Theta(z)-(z-1){\bf g}_{A,{\bf v}}^*(I-A)^{-1}\mathfrak C_{A,{\bf v}}(I-zF^*)^{-1}\notag\\
&\qquad\qquad\quad \times (\mathfrak C_{A,{\bf v}}^{-1}-F^*\mathfrak C_{A,{\bf v}}A)P_{A,{\bf v}}(I-zA^*)^{-1}{\bf g}_{A,{\bf v}},\label{feb14}
\end{align}
where we used the equality \eqref{feb11} for the second step. We next rearrange
$$
\mathfrak C_{A,{\bf v}}^{-1}-F^*\mathfrak C_{A,{\bf v}}A=F^*\mathfrak C_{A,{\bf v}}^{-1}(zI-A)
+(I-zF^*)\mathfrak C_{A,{\bf v}}^{-1}
$$
and recall computation \eqref{feb13} to get
\begin{align*}
&\mathfrak C_{A,{\bf v}}(I-zF^*)^{-1}(\mathfrak C_{A,{\bf v}}^{-1}-F^*\mathfrak C_{A,{\bf v}}A)P_{A,{\bf v}}(I-zA^*)^{-1}{\bf g}_{A,{\bf v}}\\
&=\mathfrak C_{A,{\bf v}}^{-1}
(I-zF^*)^{-1}F^*\mathfrak C_{A,{\bf v}}^{-1}(zI-A)P_{A,{\bf v}}(I-zA^*)^{-1}{\bf g}_{A,{\bf v}}\\
&\qquad + P_{A,{\bf v}}(I-zA^*)^{-1}{\bf g}_{A,{\bf v}}\\
&=\mathfrak C_{A,{\bf v}}^{-1} (I-zF^*)^{-1}F^*\mathfrak C_{A,{\bf v}}^{-1}{\bf v}\Theta(z)+P_{A,{\bf v}}(I-zA^*)^{-1}{\bf g}_{A,{\bf v}}\\
&=P_{A,{\bf v}}(I-zA^*)^{-1}{\bf g}_{A,{\bf v}},
\end{align*}
where the last step follows due to equalities $F^*\mathfrak C_{A,{\bf v}}^{-1}{\bf v}=F^*{\bf e}_1=0$ (see \eqref{1.15e} and \eqref{1.14eg}).
Substituting the latter equality into \eqref{feb14} we arrive at 
$$
G(z)R(z)=\Theta(z)-(z-1){\bf g}_{A,{\bf v}}^*(I-A)^{-1}P_{A,{\bf v}}(I-zA^*)^{-1}{\bf g}_{A,{\bf v}}.
$$
Making use of formulas \eqref{6u} and \eqref{14u} for ${\bf g}_{A,{\bf v}}$ and $\Theta(z)$ we simplify the expression on the right side
as follows:
\begin{align*}
G(z)R(z)&=1-{\bf v}^*(I-A^*)^{-1}{\bf g}_{A,{\bf v}}+z{\bf v}^*(I-zA^*)^{-1}{\bf g}_{A,{\bf v}}\\
&\quad-(z-1){\bf v}^*(I-A^*)^{-1}(I-zA^*)^{-1}{\bf g}_{A,{\bf v}}\\
&=1-{\bf v}^*(I-A^*)^{-1}\big(I-zA^*-z(I-A^*)\big)(I-zA^*)^{-1}{\bf g}_{A,{\bf v}}\\
&\quad-(z-1){\bf v}^*(I-A^*)^{-1}(I-zA^*)^{-1}{\bf g}_{A,{\bf v}}=1
\end{align*}
Thus the polynomial $G$ is indeed the reciprocal of $R$, which completes the proof of the theorem. 
\end{proof}
\begin{Rk}
{\rm Note that given a polynomial $p$ as in Theorem \ref{T:1}, there is also a power series $\widetilde{R}$ with no zeros in $\mathbb B_1$
such that $\widetilde{R}p$ is a finite Blaschke product. To see this, let us apply Theorem \ref{T:1} to the conjugate polynomial 
$p^\sharp$ to find a power series $R$ such that $R^\sharp p\sharp$ is a finite Blaschke product, and then let $\widetilde{R}=R^\sharp$.} 
\label{R:4.5}
\end{Rk}
We next apply Theorem \ref{T:1} to three particular choices of $A$ and ${\bf v}$. 
\begin{Ex}
{\rm We recall a question from Section 1: {\em given a monic polynomial $p(z)=z^n+p_{n-1}z^{n-1}+\ldots+p_1z+p_0$ with all zeros 
inside $\mathbb B_1$, find a power series $R$ without zeros such that $pR:=\Theta$ is a finite Blaschke product}. Theorem \ref{T:1} 
provides a solution: represent $p$ in terms of its companion matrix $C_p$ as in \eqref{1.15e} and then apply Theorem \ref{T:1}
with $A=C_p$ and ${\bf v}={\bf e}_1$. The formulas \eqref{9u} and \eqref{14u} give fairly explicit expressions for $R$ and $\Theta$.}
\label{E:u1}
\end{Ex}
\begin{Ex}
{\em It was shown in \cite{bolprep}, that given any collection $\bgam=(\gamma_1,\ldots,\gamma_n)\subset\mathbb H$, the pair
\begin{equation}
A=A_{\bgam}:=\begin{bmatrix}\gamma_1 &0&\ldots&&0\\ 1
&\gamma_2&0&&\\ 0&1&\ddots&\ddots&\vdots\\
\vdots&\ddots&\ddots&\ddots&0\\ 0&\ldots &0&1&\gamma_n\end{bmatrix},\quad {\bf v}={\bf e_{1,n}}:=\begin{bmatrix}1 \\ 0 \\ 0 \\ \vdots \\ 0\end{bmatrix}
\label{1.5k}
\end{equation}
is controllable, and the formula \eqref{1.15g} defines the polynomial $p=\bp_{\gamma_1}\bp_{\gamma_2}\cdots \bp_{\gamma_n}$.
Assuming that $|\gamma_k|<1$ for $k=1,\ldots,n$, we may apply Theorem \ref{T:1} with $A$ and ${\bf v}$ as in \eqref{1.5k}, to get $R$ 
such that $\bp_{\gamma_1}\bp_{\gamma_2}\cdots \bp_{\gamma_n}R$ is a finite Blaschke product. We hence partly recover 
Theorem \ref{T:jan1}, where $R$ was constructed recursively.}
\label{E:u2}
\end{Ex}
\begin{Ex}
{\rm As we have already mentioned, the general problem of constructing a finite Blaschke product with prescribed spherical divisors
easily reduces to the case where all spherical divisors are indecomposable polynomials based spherical chains 
$\bgam_k=(\gamma_{k,1},\gamma_{k,2},\ldots,\gamma_{k,n_k})$ from distinct similarity
classes. Let us take these indecomposable polynomials in the form 
\begin{equation}
P_{\bgam_k}=\bp_{\gamma_{k,1}}\bp_{\gamma_{k,2}}\cdots \bp_{\gamma_{k,n_k}}\quad \mbox{for}\quad k=1,\ldots,m.
\label{1.5b}
\end{equation}
With each polynomial $P_{\bgam_k}$, we associate the pair $(A_{\bgam_k},{\bf e}_{1,n_k})$ defined via formulas 
\eqref{1.5k}. Their direct sum $(A,{\bf v})$ defined as 
\begin{equation}
A=\begin{bmatrix}A_{\bgam_1} & & 0 \\ & \ddots & \\ 0 & & A_{\bgam_m}\end{bmatrix},\quad 
{\bf v}=\begin{bmatrix}{\bf e}_{1,n_1} \\ \vdots \\ {\bf e}_{1,n_m}\end{bmatrix}
\label{1.5m}
\end{equation}
is controllable, and the formula \eqref{1.15g} defines the polynomial 
$$
p={\bf lrcm}\{P_{\bgam_k}: \, k=1,\ldots,m\};
$$ 
see \cite{bolprep} for details. Assuming that all spherical chains $\bgam_k$ above are in $\mathbb B_1$, we may apply Theorem \ref{T:1} with $A$ and ${\bf v}$ as in \eqref{1.5m}
to get a finite Blaschke product with prescribed left spherical divisors \eqref{1.5b}.}
\label{E:u3}
\end{Ex}
\begin{Rk}
{\rm Theorem \ref{T:1} suggests a general way to construct a stable unitary realization of a given finite Blaschke product $f$.
Indeed, we can $f$ (given as in \eqref{9.6}) in the form $f=pR$, where
$p$ is a monic polynomial with all zeros inside $\mathbb B_1$ and $R\in\mathcal H_1$ has no zeros; since all similarity classes
enclosing all zeros of $f$ are known from \eqref{9.6}, the polynomial $p$ can be constructed explicitly as suggested in Theorem \ref{T:1.1}.
The polynomial $p$ has the same left zero structure as $f$, and in particular, all its zeros are contained in $\mathbb B_1$. We then
apply Theorem \ref{T:1} with any controllable pair $(A,{\bf v})$ representing $p$ in the form \eqref{1.15g} (for example we may choose
$A=C_p$ and ${\bf v}={\bf e}_1$) to come up with the finite Blaschke product $\Theta$ of the form \eqref{14u}. This $\Theta$
admits a unitary realization \eqref{feb8} and is equal to the original $f$ up to a unimodular constant factor which can be incorporated as in 
\eqref{dop7} to produce the desired realization for $f$.}
\label{R:dop}
\end{Rk}
We next discuss the uniqueness of the stable unitary realization \eqref{132}. As we have seen in the proof 
of Theorem \ref{T:1} (Step 2), two realizations define the same power series
\begin{equation}
f(z)=D+zC(I-zA)^{-1}B=D+z\widetilde{C}(I-z\widetilde{A})^{-1}\widetilde{B},
\label{dop8}
\end{equation}
if there exists an invertible matrix $V$ such that 
\begin{equation}
VA=\widetilde{A}V, \quad VB=\widetilde{B}, \quad C=\widetilde{C}V.
\label{dop9}
\end{equation}
Two realizations related as in \eqref{dop9} are called {\em similar}; if equalities \eqref{dop9} hold for some {\em unitary} $V$,
the realizations are called {\em unitarily equivalent}. In Proposition \ref{P:4.10} below, we will shows that stable unitary realizations
of finite Blaschke products are unique up to unitary equivalence. To this end we need the following auxiliary result.
\begin{lemma}
Let $A\in\mathbb H^{n\times n}$ and ${\bf v}\in\mathbb H^{n\times 1}$ satisfy $AA^*+{\bf v}{\bf v}^*=I_n$.
Then the pair $(A,{\bf v})$ is controllable if and only if $A$ is stable (i.e., $\sigma_{\bf
r}(A)\subset\mathbb B_1$).
\label{P:5.1}
\end{lemma}
\begin{proof}
By the assumption, $A$ is a contraction; hence the limit
$\Delta:={\displaystyle\lim_{k\to\infty}A^{k}A^{*k}}\succeq 0$ exists.
Clearly, $\Delta=0$ if and only if $A$ is stable. Letting $k\to\infty$ in the identity
$$
\sum_{j=0}^k A^{j}{\bf v}{\bf v}^*A^{*j}=\sum_{j=0}^\infty A^{j}(I_n-AA^*)A^{*j}=I_n-A^{k+1}A^{*(k+1)}
$$
we conclude (upon making use the notation \eqref{7u}) that
\begin{equation}
P_{A,{\bf v}}:=\sum_{j=0}^\infty A^{j}{\bf v}{\bf v}^*A^{*j}=I_n-\Delta\succeq 0.
\label{5.1}
\end{equation}
If the pair $(A,{\bf v})$ is controllable, its controllability matrix \eqref{feb1}
is invertible and hence (see \eqref{7s}),
\begin{equation}
P_{A,{\bf v}}\succeq \mathfrak C_{A,{\bf v}}\mathfrak C_{A,{\bf v}}^*\succ \varepsilon \cdot I_n \qquad\mbox{for some}\quad \varepsilon>0.
\label{5.2}
\end{equation}
It follows from the series representation for $P_{A,{\bf v}}$, that
$$
P_{A,{\bf v}}=\sum_{j=0}^{k-1} A^{j}{\bf v}{\bf v}^*A^{*j}+A^{k}P_{A,{\bf v}}A^{*k},
$$
and furthermore, by letting $k\to\infty$ on the right side, that
${\displaystyle\lim_{k\to\infty}A^{k}P_{A,{\bf v}}A^{*k}}=0$. Combining the latter with
\eqref{5.2} we conclude that $\Delta=0$ and hence $A$ is stable.

\smallskip

Conversely, if $A$ is stable then $P_{A,{\bf v}}=I_n$, by \eqref{5.1}. If the pair $(A,{\bf v})$ is not controllable,
then the columns of $\mathfrak C_{A,{\bf v}}$ are right-linearly dependent and hence, for some $k<n$,
$A^{k}{\bf v}$ belongs to the right linear span $\operatorname{{\bf span}_{\bf r}}\{A^{j}{\bf v}: \, 1\le j<k\}$. But then
$A^{m}{\bf v}$ belongs to the same linear span for all $m\ge k$ and we conclude from \eqref{5.1} that $P_{A,{\bf v}}=I_n$ is singular.
The latter contradiction completes the proof.\qed
\end{proof}
\begin{proposition}
Let us assume that equality \eqref{dop8} holds for two unitary realizations with stable matrices $A\in\mathbb H^{n\times m}$
and $\widetilde{A}\in\mathbb H^{m\times m}$. Then $n=m$ and equalities \eqref{dop9} hold for some unitary matrix $V$.
\label{P:4.10}
\end{proposition}
\begin{proof}
From the computation \eqref{133} and unitary realization formulas \eqref{dop8} we have for all $\alpha,\beta\in\mathbb B_1$,
$$
1-f^{\bl}(\alpha)\overline{f^{\bl}(\beta)}=\Upsilon(\alpha)\Upsilon(\beta)^*-\alpha\Upsilon(\alpha)\Upsilon(\beta)^*\overline{\beta}
=\widetilde{\Upsilon}(\alpha)\widetilde{\Upsilon}(\beta)^*-\alpha\widetilde{\Upsilon}(\alpha)\widetilde{\Upsilon}(\beta)^*\overline{\beta},
$$
where $\Upsilon(\gamma)$ is defined in \eqref{ups} and where $\Upsilon(\gamma):=\sum_{k\ge 0}\gamma_kCA^k$. Then we have
$$
\Upsilon(\alpha)\Upsilon(\beta)^*-\widetilde{\Upsilon}(\alpha)\widetilde{\Upsilon}(\beta)^*=
\alpha(\Upsilon(\alpha)\Upsilon(\beta)^*-\widetilde{\Upsilon}(\alpha)\widetilde{\Upsilon}(\beta)^*)\overline{\beta}
$$
which implies, since $|\alpha|<1$ and $|\beta|<1$,
$$
\Upsilon(\alpha)\Upsilon(\beta)^*=\widetilde{\Upsilon}(\alpha)\widetilde{\Upsilon}(\beta)^*\quad\mbox{for all}\quad \alpha,\beta\in\mathbb B_1.
$$
Hence, the transformation $V$ defined by the formula
\begin{equation}
V: \, \Upsilon(\beta)^*\to \widetilde{\Upsilon}(\beta)^*
\label{dop11}
\end{equation}
extends by linearity to an isometry $V$ from 
$$
\mathcal D_V=\operatorname{\bf span}_{\bf r}\{\Upsilon(\beta)^*: \, \beta\in\mathbb B_1\}\quad\mbox{onto}\quad 
\mathcal R_V=\operatorname{\bf span}_{\bf r}\{\widetilde{\Upsilon}(\beta)^*: \, \beta\in\mathbb B_1\}. 
$$
Observe that $C^*=\Upsilon(0)^*$
belongs to $\mathcal D_V$. Therefore, for any $\beta\in\mathbb B_1$, the vector $\Upsilon(\beta)^*-C^*=A^*\Upsilon(\beta)^*\overline\beta$
belongs to $\mathcal D_V$ and hence $A^*\Upsilon(\beta)^*\in \mathcal D_V$ for all nonzero $\beta\in\mathbb B_1$. Since $\mathcal D_V$ is closed,
we have $A^*\Upsilon(0)^*=A^*C^*\in \mathcal D_V$. Repeating the above argument we conclude that $A^{*j}C^*\in\mathcal D_V$ for all $j\ge 0$.
Since the matrix $\sbm{A & B \\ C & G}$ is unitary, we have $A^*A+C^*C=I_n$. Since $A$ is stable, the pair $(A^*,C^*)$ is controllable, 
by Lemma \ref{P:5.1}. Then $\mathbb H^{n\times 1}=\operatorname{\bf span}_{\bf r}\{C^*,A^*C^*,\ldots,A^{*(n-1)}C^*\}\subseteq \mathcal D_V$.
Since the converse inclusion is obvious, we have $\mathcal D_V=\mathbb H^{n}$. 

\smallskip

Since the rightmost realization in \eqref{dop8} is also unitary and stable, we conclude that the pair $(\widetilde{A}^*,\widetilde{C}^*)$ is
controllable and hence, $\mathcal R_V=\mathbb H^{m}$. Therefore, $V$ is a surjective isometry, that is, $V$ is unitary and $n=m$. 

\smallskip

To verify equalities \eqref{dop9}, we first let $\beta=0$ in \eqref{dop11} to get $VC^*=\widetilde{C}^*$, which is equivalent to
$C=\widetilde{C}V$. Again making use of \eqref{dop11}, we write
$$\eqref{dop8}
VA^*\Upsilon(\beta)^*\overline\beta=V(\Upsilon(\beta)^*-C^*)=\widetilde{\Upsilon}(\beta)^*-\widetilde{C}^*=
\widetilde{A}^*\widetilde{\Upsilon}(\beta)^*\overline\beta=\widetilde{A}^*V\Upsilon(\beta)^*\overline\beta,
$$
from which it follows, since the pair $(A^*,C^*)$ is controllable, that $VA^*=\widetilde{A}^*V$, which is equivalent to 
$VA=\widetilde{A}V$. Finally, we equate the adjoints of the corresponding coefficients in \eqref{dop8} to get
$$
B^*A^{*j}C^*=\widetilde{B}^*\widetilde{A}^{*j}\widetilde{C}^*=\widetilde{B}^*\widetilde{A}^{*j}VC^*=\widetilde{B}^*VA^{*j}C^*
$$
for all $j\ge 0$, and hence, $B^*=\widetilde{B}^*V$ which completes the verification of \eqref{dop9}. \qed
\end{proof}
Now we can complete the proof of Theorem \ref{L:feb2}.
\begin{corollary}
If the matrix $\sbm{A & B \\ C & D}\in\mathbb H^{(n+1)\times(n+1)}$ with $D\in\mathbb H$ is unitary and 
$\sigma_{\bf r}(A)\subset\mathbb B_1$, then the power series \eqref{132} is a finite Blaschke product of degree $n$.
\label{C:4.7}
\end{corollary}
\begin{proof}
We have already shown that $f$ is a finite Blaschke product. Let $\deg f=m$. By the converse statement in Theorem \ref{L:feb2},
$f$ admits a unitary realization \eqref{dop7} with the stable matrix $A_m\in\mathbb H^{m\times m}$. By Proposition \ref{P:4.10}, these 
realizations are unitarily equivalent and in particular, $m=n$.\qed
\end{proof}

\section{Further characterizations of finite Blaschke products}

Theorem \ref{T:8.2} specifies parts (1) and (2) in Theorem \ref{T:3.1} to the case of finite Blaschke products. 
In this section we will present "finite Blaschke product" analogs of other three parts. 
The next result in \cite{acsbook} specifies part (4). The proof below is somewhat different from that in \cite{acsbook}.
\begin{theorem}
A power series  $f\in \bH[[z]]$ is a Blaschke product of degree $n$ if and only if
$\|fh\|_{H^2(\mathbb B_1)}=\|h\|_{H^2(\mathbb B_1)}$ for all $h\in H^2(\mathbb B_1)$ 
and the right $\mathbb H$-submodule $f\cdot  H^2(\mathbb B_1)\subseteq H^2(\mathbb B_1)$ has codimension $n$.
\label{T:3.1c}
\end{theorem}
\begin{proof}
The isometric multiplier property $\|{\bf b}_\alpha h\|=\|h\|$ of the Blaschke factor reduces to the complex case as follows.
By Lemma \ref{L:j1}, any $h\in H^2(\mathbb B_1)$ can be represented as 
$h=h_1+h_2\beps$ with some $\beps\in\mathbb C_\alpha^\perp$ and $h_1,h_2\in H^2(\mathbb C_\alpha\cap\mathbb B_1)$.
By the definition \eqref{dop4} of the norm in $H^2(\mathbb B_1)$, it follows that 
$$
\|h\|^2_{H^2(\mathbb B_1)}=\|h_1\|^2_{H^2(\mathbb C_\alpha\cap\mathbb B_1)}+\|h_2\|^2_{H^2(\mathbb C_\alpha\cap\mathbb B_1)}.
$$
Since ${\bf b}_{\alpha}h_1$ and ${\bf b}_{\alpha}h_1$ also belong to $H^2(\mathbb C_\alpha\cap\mathbb B_1)$, we have,
by the isometric property of a complex Blaschke factor,
\begin{align*}
\|{\bf b}_{\alpha}h\|^2_{H^2(\mathbb B_1)}&=\|{\bf b}_{\alpha}h_1\|^2_{H^2(\mathbb C_\alpha\cap\mathbb B_1)}+
\|{\bf b}_{\alpha}h_2\|^2_{H^2(\mathbb C_\alpha\cap\mathbb B_1)}\\
&=\|h_1\|^2_{H^2(\mathbb C_\alpha\cap\mathbb B_1)}+\|h_2\|^2_{H^2(\mathbb C_\alpha\cap\mathbb B_1)}=\|h\|^2_{H^2(\mathbb B_1)}.
\end{align*}
The general case follows recursively: 
$$
\|{\bf b}_{\alpha_1}{\bf b}_{\alpha_2}\cdots {\bf b}_{\alpha_n}h\|=\|{\bf b}_{\alpha_2}\cdots {\bf b}_{\alpha_n}h\|=\ldots=
\|{\bf b}_{\alpha_n}h\|=\|h\|. 
$$
To complete the proof of the "only if" part we recall the $n$-elements Takenaka basis \eqref{dop5} for $(f\cdot H^2(\mathbb B_1))^\perp$.

\smallskip

For the "if" part, we start with the right $\mathbb H$-submodule $\mathcal N:=(f\cdot H^2(\mathbb B_1))^\perp$, which is necessarily invariant
under the backward-shift operator $R_0$ defined in \eqref{dop7a}. Choosing a (right) orthonormal basis of $\mathcal N$ we let $A$ and $C$ 
to denote the matrices of the operator $R_0$ and of the functional $E: \, h\to h_0$ with respect to this basis.  
Then the pair $(A^*,C^*)$ turns out to be controllable and satisfy $A^*A+C^*C=I_n$, while $\mathcal N$ is described by 
$$
\mathcal N=\bigg\{C(I-zA)^{-1}x=\sum_{k=0}^\infty (CA^kx) z^k: \, x\in\mathbb H^n\bigg\}.
$$
Extending the isometric matrix $\sbm{A \\ C}$ to a unitary matrix $\sbm{A & B \\ C& D}$, we next consider the power series 
$\Theta(z)=D+zC(I_n-zA)^{-1}B$. Since the pair $(A^*,C^*)$ is controllable, we conclude (by combining Lemma \ref{P:5.1} and Theorem \ref{L:feb2})
that $\Theta$ is a Blaschke product of degree $n$. Furthermore. it follows from \eqref{5.1} that 
\begin{equation}
P_{A^*,C^*}:=\sum_{k=0}^\infty A^{*k}C^*CA^{k}=I.
\label{6.0}
\end{equation}
We next compute the inner product \eqref{dop6} of a  
generic element $q_x(z)=C(I-zA)^{-1}x$ in $\mathcal N$ and $\Theta(z)z^m$ for some fixed $m\ge 0$:
\begin{equation}
\langle q_x, \, z^m\Theta \rangle_{H^2(\mathbb B_1)}=(D^*C+B^*
\bigg(\sum_{k=0}^\infty A^{*k}C^*CA^{k}\bigg)A)A^mx=0,
\label{6.0a}
\end{equation}
where we used \eqref{6.0} and equality $D^*C+B^*A=0$ for the last step. Since the multiplication 
operator $M_\Theta: h\to \Theta h$ is isometric on $H^2(\mathbb B_1)$, it follows from \eqref{6.0a} that
$\mathcal N^\perp$ is contained in the right $\mathbb H$-submodule $\Theta\cdot H^2(\mathbb B_1)$.
We next observe that any $h(z)=\sum_{k\ge 0}h_kz^k\in H^2(\mathbb B_1)$ can be decomposed as 
\begin{equation}
h(z)=C(I-zA)^{-1}u+\Theta(z)\cdot \widetilde{h}(z),
\label{6.0b}
\end{equation}
where $u\in\mathbb H^n$ and $\widetilde{h}\in \in H^2(\mathbb B_1)$ are given by  
$$
u=\sum_{k=0}^\infty A^{*k}C^*h_{k}\quad\mbox{and}\quad \widetilde{h}(z)=(M_\Theta^*h)(z)=
D^*h(z)+\sum_{k=0}^\infty B^*A^{*k}C^*(R_0^{k+1}h)(z)
$$
where $R_0$ is the backward shift operator \eqref{dop7a}. The convergence of the series representing $u$ follows 
by Cauchy inequality and \eqref{6.0}. Since we already know that $\mathcal N^\perp\subseteq \Theta\cdot H^2(\mathbb B_1)$, 
it now follows from \eqref{6.0b} that actually 
\begin{equation}
f\cdot H^2(\mathbb B_1)=\mathcal N^\perp=\Theta\cdot H^2(\mathbb B_1).
\label{6.0c}
\end{equation}
Due to the isometric multiplier property of both $f$ and $\Theta$, the orthogonal projection of $H^2(\mathbb B_1)$ onto 
$\mathcal N^\perp$ can be written in two ways:
$$
P_{\mathcal N^\perp}=M_\Theta M_\Theta^*=M_fM_f^*.
$$
Applying the latter equality to the constant ${\bf 1}$ gives the identity 
$\Theta(z)\overline{\Theta}_0=f(z)\overline{f}_0$. Due to \eqref{6.0c}, $f=\Theta g$ for some $g\in H^2(\mathbb B_1)$. Combining the two
latter equalities gives $\Theta(z)\overline{\Theta}_0=\Theta(z)g(z)\overline{g}_0\overline{\Theta}_0$, which implies 
$g(z)\overline{g}_0\equiv 1$. Therefore, $g$ is constant (say, $\phi$) and hence, $f=\Theta \phi$. By the isometric multiplier property.  
$1=\|f\cdot{\bf 1}\|=\|\Theta \phi {\bf 1}\|=|\phi|$. Thus, $|\phi|=1$ and hence $f=\Theta \phi$ is a Blaschke product of degree $n$.

\smallskip

If $\Theta_k$ is the first nonzero coefficient of $\Theta$ (i.e., $\Theta(z)=z^k\widetilde{\Theta}(z)$, where $\widetilde{\Theta}$ is a 
Blaschke product of degree $n-k$, then it follows from \eqref{6.0c} that $f(z)=z^n\widetilde{f}(z)$ for some isometric multiplier 
$\widetilde f$ and moreover, $\widetilde{f}\cdot H^2(\mathbb B_1)=\widetilde{\Theta}\cdot H^2(\mathbb B_1)$. By the case considered above,
$\widetilde{f}$ is equal to $\widetilde{\Theta}$ up to a right unimodular factor and hence, $f$ is a Blaschke product of degree $n$.\qed
\end{proof}
Let us now consider the space $H^2(\mathbb B_1)$ as the left Hilbert $\mathbb H$-module with inner product
$$
[ h, \, g]_{H^2(\mathbb B_1)}=\sum_{k=0}^\infty h_k \overline{g}_k,\quad\mbox{where}\quad h(z)=\sum_{k=0}^\infty h_kz^k, \; \;
g(z)=\sum_{k=0}^\infty g_kz^k.
$$
Then the following result (the counter-part of the characterization (4) in Theorem \ref{T:3.1})
can be easily derived from Theorem \ref{T:3.1c} by using power series conjugation \eqref{9}.
\begin{theorem}
A power series  $f\in \bH[[z]]$ is a Blaschke product of degree $n$ if and only if
$\|hf\|_{H^2(\mathbb B_1)}=\|h\|_{H^2(\mathbb B_1)}$ for all $h\in H^2(\mathbb B_1)$ 
and the left $\mathbb H$-submodule $\cdot H^2(\mathbb B_1) f$ of $H^2(\mathbb B_1)$ has codimension $n$.
\label{T:3.1d}
\end{theorem}
We finally characterize finite Blaschke products in terms of their coefficients. A necessary condition was pointed out
in part (5) of Theorem \ref{T:3.1}.
\begin{theorem}
A power series $f\in \bH[[z]]$ is a Blaschke product of degree $n$ if and only if the associated matrix ${\bf T}_{k}^f$ \eqref{1.4} is 
contractive for all $k\ge 1$ and 
\begin{equation}
\operatorname{rank} {\bf P}^f_k={\rm min} (k,n), \quad\mbox{where}\quad {\bf P}^f_k:=I_k-{\bf T}_{k}^f({\bf T}_{k}^{f})^*.
\label{6.1a}
\end{equation}
\label{T:3.1b}
\end{theorem}
\begin{proof}
If $f$ is a Blaschke product of degree $n$, we can take it in the form \eqref{132} for some   
unitary matrix $\sbm{A & B \\ C & D}$ with $D\in\mathbb H$ and stable $A\in\mathbb H^{n\times n}$.
Then we have 
\begin{equation}
f_0=D\quad\mbox{and}\quad f_j=CA^{j-1}B\quad\mbox{for all}\quad j\ge 1.
\label{6.17}
\end{equation}
Substituting the latter expression into the formula \eqref{6.1a} for ${\bf P}^f_k$ leads
(upon simple manipulations based on equalities \eqref{ups1}) to the representation
\begin{equation}
{\bf P}^f_k=I_k-{\bf T}_{k}^f({\bf T}_{k}^{f})^*=\begin{bmatrix}C \\ CA \\ \vdots \\ CA^{k-1}\end{bmatrix}
\begin{bmatrix}C^* & A^*C^* & \ldots & A^{*(k-1)}C^*\end{bmatrix}.
\label{6.1}
\end{equation} 
Since the vectors $C^*, A^*C^*,\ldots,A^{*(n-1)}C^*\in\mathbb H^n$ are right linearly independent 
(since the pair $(A^*,C^*)$ is controllable, by Lemma \ref{P:5.1}), we have
$$
\operatorname{rank}{\bf P}^f_k=\operatorname{rank}\begin{bmatrix}C^* & A^*C^* & \ldots & A^{*(k-1)}C^*\end{bmatrix}={\rm min} (k,n).
$$
which completes the proof of the ``only if" part of the theorem. The ``if" part will proven below in a seemingly stronger form.\qed
\end{proof}
In what follows, we drop the superscript $f$ and consider the lower triangular Toeplitz matrix ${\bf T}_n$ and the associated 
matrix ${\bf P}_n=I_n-{\bf T}_n{\bf T}_n^*$ as structured matrices associated with a given finite sequence 
$(f_0,\ldots, f_{n-1})\in\mathbb H$.  The matrix ${\bf P}_{n+r}=I_n-{\bf T}_{n+r}{\bf T}_{n+r}^*$ associated with the extended 
sequence $(f_0,\ldots, f_{n+r-1})$ is called a {\em structured extension} of ${\bf P}_n$.
Writing 
\begin{equation}
{\bf T}_{n+1}=\begin{bmatrix} {\bf T}_n & 0 \\ X_n^* & f_0\end{bmatrix},\quad\mbox{where}\quad
X_n=\sbm{\overline{f}_n \\ \vdots \\ \overline{f}_1},
\label{7.5}
\end{equation}
we then have the block decomposition 
\begin{equation}
{\bf P}_{n+1}:=I_{n+1}-{\bf T}_{n+1}{\bf T}_{n+1}^*=
\begin{bmatrix} {\bf P}_n & -{\bf T}_n X_n\\
-X_n^*{\bf T}^*_n\quad & 1-|f_0|^2-X_n^*X_n\end{bmatrix}.
\label{7.5a}
\end{equation}
We recall the matrix $F$ given in \eqref{1.14eg} and the columns ${\bf e}_1,\ldots,{\bf e_n}$ of the matrix $I_n$.
\begin{theorem}
Given $f_0,\ldots,f_n\in\mathbb H$, let us assume that ${\bf P}_{n+1}\succeq 0$ and that
\begin{equation}
\operatorname{rank} {\bf P}_{n+1}=\operatorname{rank} {\bf P}_n=n.
\label{7.6}
\end{equation}
Then the power series 
\begin{equation}
G(z)=f_0+z{\bf e}_1^*\big(I_n-z \big(F^*-{\bf e}_n X_n^*{\bf T}_n^*{\bf P}_n^{-1}\big)\big)^{-1}Y_n,\quad\mbox{where}\quad
Y_n=\sbm{f_1 \\ f_2 \\ \vdots \\ f_n},
\label{7.7}
\end{equation}
is a Blaschke product of degree $n$ and is the unique Schur-class power series with the first $n+1$ coefficients equal to
$f_0,f_1,\ldots,f_n$.
\label{T:7.1}
\end{theorem}
\begin{proof}
We first note that the state space matrix in the realization \eqref{7.7} is the adjoint of the companion matrix 
\begin{equation}
C_{g}=F-{\bf P}_n^{-1}{\bf T}_nX_n{\bf e}_n^*=\begin{bmatrix}{\bf e}_2 & {\bf e}_3 & \ldots & {\bf e}_n & {\bf P}_n^{-1}{\bf T}_nX_n\end{bmatrix}
\label{7.8}
\end{equation}
of the polynomial $g(z)=z^n+\begin{bmatrix}1 & z & \ldots & z^{n-1}\end{bmatrix} {\bf P}_n^{-1}{\bf T}_nX_n$. It is readily seen from \eqref{7.8}
that $C_{g}{\bf e}_j={\bf e}_{j+1}$ for $j=1,\ldots,n-1$, from which we get recursively 
$$
C_{g}^j{\bf e}_1={\bf e}_{j+1}\quad\mbox{for}\quad j=1,\ldots,n-1.
$$
Taking the latter equalities into account, along with obvious equalities ${\bf e}_{j}*Y=f_j$,
we now have from \eqref{7.7} and \eqref{7.8}
\begin{align}
G(z)=f_0+\sum_{j=0}^\infty {\bf e}_1^*C_{g}^{*j}Y_n
&=f_0+\sum_{j=0}^{n-1}{\bf e}_{j+1}^*Y_nz^{j+1}+\sum_{j=n+1}^\infty {\bf e}_1^*C_{g}^{*j-1}Y_nz^j\notag  \\
&=f_0+f_1z+\ldots+f_nz^n+\sum_{j=n+1}^\infty {\bf e}_1^*C_{g}^{*j-1}Y_nz^j,\label{7.9a}
\end{align}
confirming that $G$ indeed has the desired coefficients. We next verify the equalities
\begin{equation}
 C_g^* {\bf P}_nC_g +Y_nY_n^*={\bf P}_n,\quad C_g^*{\bf P}_n{\bf e}_1+Y_n\overline{f}_0=0, \quad {\bf e}^*_1{\bf P}_n{\bf e}_1+|f_0|^2=1.
\label{7.9}
\end{equation}
Here we use the degeneracy condition \eqref{7.6} which implies that the Schur complement of the block ${\bf P}_n$
in \eqref{7.5a} equals zero:
$$
1-|f_0|^2-X_n^*X_n-X_n{\bf T}^*_n{\bf P}_n^{-1}{\bf T}_nX_n=0.
$$
We now use \eqref{7.8} and the last equality to compute
\begin{align}
C_g^* {\bf P}_nC_g&=\big(F^*-{\bf e}_n X_n^*{\bf T}_n^*{\bf P}_n^{-1}\big){\bf P}_n
\big(F-{\bf P}_n^{-1}{\bf T}_nX_n{\bf e}_n^*\big)\notag\\
&=F^*{\bf P}_nF-{\bf e}_n X_n^*{\bf T}_n^*F-F^*{\bf T}_nX_n{\bf e}_n^*+{\bf e}_n\big(1-|f_0|^2-X_n^*X_n\big){\bf e}_n^*. 
\label{7.18}
\end{align}
The rightmost term on the right side can be written in the block-matrix form as 
\begin{equation}
{\bf e}_n\big(1-|f_0|^2-X_n^*X_n\big){\bf e}_n^*=\begin{bmatrix}0 & 0 \\ 0 & \quad 1-|f_0|^2-X_{n-1}^*X_{n-1}-|f_n|^2\end{bmatrix}.
\label{7.19}
\end{equation}
We next observe block matrix representations (conformal with that in \eqref{7.19})
\begin{equation}
F^*{\bf P}_nF=\begin{bmatrix}{\bf P}_{n-1}-Y_{n-1}Y_{n-1}^* \; & 0 \\ 0 & 0\end{bmatrix}\quad\mbox{and}\quad
F^*{\bf T}_n=\begin{bmatrix}Y_{n-1} \;  & {\bf T}_{n-1} \\ 0 & 0\end{bmatrix},
\label{7.20}
\end{equation}
which follow from the explicit formulas for $F$, ${\bf P}_n$, ${\bf T}_n$. Then we also have from \eqref{7.5}
\begin{equation}
F^*{\bf T}_nX_n{\bf e}_n^*=\begin{bmatrix}Y_{n-1} & \; {\bf T}_{n-1} \\ 0 & 0\end{bmatrix}\begin{bmatrix}0 & \overline{f}_n  \\ 0 & X_{n-1}
\end{bmatrix}=\begin{bmatrix}0 & \; Y_{n-1}\overline{f}_n+ {\bf T}_{n-1}X_{n-1} \\ 0& 0 \end{bmatrix}.
\label{7.21}
\end{equation}
Substituting representations \eqref{7.19}-\eqref{7.21} into the right side of \eqref{7.18} leads us to 
\begin{align}
C_g^* {\bf P}_nC_g&=\begin{bmatrix}{\bf P}_{n-1}-Y_{n-1}Y_{n-1}^* & -Y_{n-1}\overline{f}_n- {\bf T}_{n-1}X_{n-1} \\ 
-f_n Y^*_{n-1}- X_{n-1}^*{\bf T}_{n-1}^* \; & 1-|f_0|^2-X_{n-1}^*X_{n-1}-|f_n|^2\end{bmatrix}\notag\\
&=\begin{bmatrix}{\bf P}_{n-1} & - {\bf T}_{n-1}X_{n-1} \\
- X_{n-1}^*{\bf T}_{n-1}^* \; & 1-|f_0|^2-X_{n-1}^*X_{n-1}\end{bmatrix}
-\begin{bmatrix}Y_{n-1} \\ f_n\end{bmatrix}\begin{bmatrix}Y^*_{n-1}\overline{f}_n\end{bmatrix}\notag\\
&= {\bf P}_n-Y_nY_{n}^*,\notag
\end{align}
which confirms the first equality in \eqref{7.9}. Observe from the second representation in \eqref{7.20} that 
$$
F^*{\bf T_n}{\bf e}_1+{\bf e}_n f_n=\begin{bmatrix}Y_{n-1} \\ f_n\end{bmatrix}=Y_n.
$$
Using the latter formula along with \eqref{7.8}, we verify the second equality in \eqref{7.9}: 
\begin{align*}
C_g^*{\bf P}_n{\bf e}_1&=F^*{\bf P}_n{\bf e}_1-{\bf e}_n X_n^*{\bf T}_n^*{\bf e}_1\\
&=F^*(I_n-{\bf T_n}{\bf T_n}^*){\bf e}_1-{\bf e}_n X_n^*{\bf e}_1\overline{f}_0\\
&=-F^*{\bf T_n}{\bf e}_1\overline{f}_0-{\bf e}_nf_n\overline{f}_0=-Y_n\overline{f}_0,
\end{align*}
The third equality in \eqref{7.9} is clear from \eqref{6.1} and \eqref{1.4}. 
Equalities \eqref{7.9} can be written in the matrix form as 
$$
\begin{bmatrix} C_g^* & Y \\ {\bf e}^*_1 & f_0 \end{bmatrix}\begin{bmatrix}{\bf P}_n & 0 \\ 0 & 1\end{bmatrix}
\begin{bmatrix} C_g &{\bf e}_1 \\ Y^* & \overline{f}_0 \end{bmatrix}=\begin{bmatrix}{\bf P}_n & 0 \\ 0 & 1\end{bmatrix}
$$
which tells us that the matrix 
$$
\begin{bmatrix} A & B \\ C & D\end{bmatrix}=\begin{bmatrix} {\bf P}_n^{-\frac{1}{2}}  & 0 \\ 0 & 1\end{bmatrix}
\begin{bmatrix} C_g^* & Y \\ {\bf e}^*_1 & f_0\end{bmatrix}
\begin{bmatrix} {\bf P}_n^{\frac{1}{2}}& 0 \\ 0 & 1\end{bmatrix}
$$ 
is unitary. Thus, he power series $G(z)$ admits a unitary realization \eqref{132} (which is similar to the realization \eqref{7.7}).
 Since the pair $(C_g,{\bf e}_1)$ is controllable, the similar pair $(A^*,C^*)$ is controllable as well.
Therefore $\sigma_{\bf r}(A)\subset\mathbb B_1$, by Lemma \ref{P:5.1}. Hence, $G$ is a Blaschke product of degree $n$, by Theorem \ref{L:feb2}. 

\smallskip

The uniqueness of a Schur-class power series subject to condition \eqref{7.9a} can be derived certain results on structured extensions 
of Hermitian matrices (see e.g. \cite[Section 2.1]{boloam}). Indeed, if $f\in\mathcal S_{\mathbb H}$ has prescribed
first $n+1$ coefficients $f_0,\ldots,f_n$, we can use further coefficients to get positive semidefinite structured extensions
${\bf P}_{n+r}=\sbm{{\bf P}_n & * \\ * &*}\succeq 0$ of the given ${\bf P}_{n}\succ 0$ for all $r\ge 1$. By mimicking complex-setting 
computations from \cite{boloam}, it follows that the Schur complement of ${\bf P}_n$ in ${\bf P}_{n+r}$ is congruent to the matrix ${\bf P}^h_{r}$ 
of the same structure, i.e., 
$$
{\bf P}^h_{r}=I_{r}-{\bf T}^h_{r}{\bf T}^{h*}_{r},
$$ 
but based on the sequence $\{h_0,\ldots,h_{r-1}\}$. The latter sequence in turn, uniquely recovers $f_n,\ldots,f_{n+r-1}$ for all $r\ge 1$. 
Furthermore, ${\bf P}_{n+1}$ is singular if and only if $|h_0|=1$. On the other hand,
${\bf P}_{n+r}\succeq 0$ for $r\ge 1$ if and only if ${\bf P}^h_{r}\succeq 0$. If $|h_0|=1$, the latter is possible only if $h_j=0$ for $j\ge 1$,
in which case ${\bf P}^h_{r}=0$ and
$$
\operatorname{rank} {\bf P}_{n+r}=\operatorname{rank} {\bf P}_{n}+\operatorname{rank} {\bf P}^h_{r}=\operatorname{rank} {\bf P}_{n}\quad
\mbox{for}\quad r\ge 1.
$$
Putting all pieces together we conclude that under assumptions \eqref{7.6}, for each $r\ge 2$, the positive semidefinite structured extension 
${\bf P}_{n+r}$ of ${\bf P}_{n+1}$ is unique and is necessarily based on the elements $f_{n+1}, f_{n+2}, ...$ corresponding to parameters $h_j=0$ 
for $j\ge 1$. The formula \eqref{7.9a} provides  a representation formula 
$$
f_{n+r}={\bf e}_1(\big(F^*-{\bf e}_n X_n^*{\bf T}_n^*)^{n+r-1}Y_n
$$
for these elements.\qed
\end{proof}
Now the "if" part in Theorem \ref{T:3.1b} follows immediately. Indeed, if ${\bf T}_k^f$ is a contraction for all $k\ge 1$, then $f\in\mathcal S_{\mathbb H}$, 
by Theorem \ref{T:3.1}. Due to condition \eqref{6.1a}, there is only one element in $S_{\mathbb H}$ with the first coefficients equal $f_0,\ldots,f_n$.
The Blaschke product $G$ of degree $n$ given in \eqref{7.7} is such an element. Therefore, $f=G$. 

\bibliographystyle{amsplain}

\end{document}